\documentclass[12pt]{amsart}
\usepackage{amsfonts, amsthm, amsmath, amssymb}
\usepackage{hyperref}
\hypersetup{colorlinks=false}

\usepackage{color}

\newtheorem{theo}{Theorem}[section]
\newtheorem{cor}[theo]{Corollary}
\newtheorem{prop}[theo]{Proposition}
\newtheorem{lemma}[theo]{Lemma}
\newtheorem{hypothesis}[theo]{Hypothesis}
\theoremstyle{remark}
\newtheorem{remark}[theo]{Remark}

\theoremstyle{definition}
\newtheorem{defi}[theo]{Definition}

\numberwithin{equation}{section}

\def \Prim {\operatorname{Prim}}
\def \tr {\operatorname{tr}}
\def \Frob {\operatorname{Frob}}
\def \sign {\operatorname{sgn}}
\def \univ {\operatorname{univ}}
\def \Sym {\operatorname{Sym}}
\def \Witt {\operatorname{Witt}}
\def \det {\operatorname{det}}
\def \Hom {\operatorname{Hom}}
\def \rp {r}
\def \rm {{\tilde{r}}}
\begin{document}

\title[A representation theory approach to integral moments]{A representation theory approach to integral moments of $L$-functions over function fields}
\author{Will Sawin}

\maketitle

\begin{abstract}We propose a new heuristic approach to integral moments of $L$-functions over function fields, which we demonstrate in the case of Dirichlet characters ramified at one place (the function field analogue of the moments of the Riemann zeta function, where we think of the character $n^{it}$ as ramified at the infinite place). We represent the moment as a sum of traces of Frobenius on cohomology groups associated to irreducible representations. Conditional on a hypothesis on the vanishing of some of these cohomology groups, we calculate the moments of the $L$-function and they match the predictions of the CFKRS recipe \cite{CFKRS}. 

In this case, the decomposition into irreducible representations seems to separate the main term and error term, which are mixed together in the long sums obtained from the approximate functional equation, even when it is dyadically decomposed. This makes our heuristic statement relatively simple, once the geometric background is set up.  We hope that this will clarify the situation in more difficult cases like the $L$-functions of quadratic Dirichlet characters to squarefree modulus. There is also some hope for a geometric proof of this cohomological hypothesis, which would resolve the moment problem for these $L$-functions in the large degree limit over function fields.
\end{abstract}

\section{Introduction}

The Conrey-Farmer-Keating-Rubinstein-Snaith heuristics give precise conjectures for the distribution of special values of $L$-functions in certain families \cite{CFKRS}. They were extended to function fields in~\cite{AK}. Certain constants appearing in these predictions can be related to statistics of random matrices.

While these are conjectures in general, they are known for many families up to an error term of $O(1/\sqrt{q})$ in the function field setting (e.g.~\cite{KatzSarnak},~\cite{WVQKR},~\cite{SQRT}). This error term hides everything but the random matrix term. However, the random matrix term appears in a particularly natural way. In the function field setting, the $L$-functions are equal to characteristic polynomials of the matrices giving the action of Frobenius elements on a certain Galois representation, and these matrices are random in a precise technical sense \cite[Theorem 3.5.3]{weil-ii}.

We are not able today to prove the full conjecture of \cite{CFKRS} over function fields for any family of $L$-functions using the geometric approach initiated by~\cite{KatzSarnak}. However, we propose a middle ground. Using the machinery of \'{e}tale cohomology, and in particular the interpretation of $L$-functions via representations of monodromy groups, we will describe a new heuristic which matches the predictions of \cite{CFKRS}. However, while the heuristics of \cite{CFKRS} require multiple manipulations, that do not make sense on their own, we will make a single assumption on vanishing of cohomology groups, which could well be true. This assumption also makes predictions for other problems, such as the variance of the divisor function in short intervals.

In this paper, we describe this heuristic, and verify its relationship to \cite{CFKRS}, only for the ``short interval" family of characters:

\begin{defi} For $n$ a natural number and $\mathbb F_q$ a finite field, consider a primitive even Dirichlet character $\psi: \left( \mathbb F_q[x]/x^{n+1} \right)^\times \to \mathbb C^{\times}$. Here ``primitive" means that the character is nontrivial on elements congruent to $1$ mod $x^n$, and ``even" means that it is trivial on $\mathbb F_q^\times$. Define a function $\chi$ on monic polynomials in $\mathbb F_q[T]$ by, for $f$ monic of degree $d$, \[ \chi(f) = \psi ( f(x^{-1}) x^{d} ) .\] It is easy to see that $\chi$ depends only on the $n+1$ leading terms of $f$. Let $S_{n,q}$ be the set of characters $\chi$ arising from primitive even Dirichlet characters $\psi$ in this way. Because there are $q^{n} $ even Dirichlet characters of which $q^{n-1}$ are imprimitive, this set has cardinality $q^{n}- q^{n-1}$. 

For $\chi\in S_{n,q}$, form  the associated $L$-functions \[L(s,\chi) = \sum_{ \substack{ f\in \mathbb F_q[T] \\ \textrm{monic}}} \chi(f)  |f|^{-s}\] where $|f| = q^{\deg f}$, with functional equation \[L(s,\chi) = \epsilon_\chi  q^{ (n-1) (1/2-s) } L(1-s, \overline{\chi}) \] for a unique $\epsilon_\chi \in \mathbb C$. \end{defi}

Let $\Prim_n = \mathbb A^n - \mathbb A^{n-1}$.  Katz \cite[\S2 \& \S3]{WVQKR} defined an explicit bijection between $\Prim_n(\mathbb F_q)$ and $S_{n,q}$. We will reproduce the precise formula in Definition~\ref{katz-bijection} below, but as the details are not relevant to the big picture, we will leave it as a black box here. 

\begin{defi}\label{L-univ} Let $L_{\univ}$ be the unique lisse sheaf of rank $n-1$ on $\Prim_n$ such that for a point $y \in \Prim_n(\mathbb F_q)$ corresponding to a character $\chi$ under the correspondence of Definition \ref{katz-bijection}, we have the identity \begin{equation}\label{L-univ-function} \det( 1- q^{-s} \Frob_{q} , L_{\univ,y}) = L(s,\chi)\end{equation}  between the characteristic polynomial of Frobenius acting on the stalk of $L_{\univ}$ at $y$ and the $L$-function of $\chi$. Katz proves the existence of this lisse sheaf by an explicit construction \cite[Lemma 4.1]{WVQKR}. \end{defi}

\begin{hypothesis}\label{hypo} Let $n,\rp,\rm,w$ be natural numbers with $0 \leq w \leq n$.

Let $\mathcal F$ be an irreducible lisse $\mathbb Q_\ell$-sheaf on $\Prim_{n, \overline{\mathbb F}_q}$ that appears as a summand of \[\det(L_{\univ})^{-\rm} \otimes \bigotimes_{i=1}^{\rp+\rm} \wedge^{d_i} (L_{\univ} )\] for some $0 \leq d_1,\dots,d_{\rp+\rm} \leq n-1$, but which does not appear as a geometric summand of $L_{\univ}^{\otimes a}  \otimes L_{\univ}^{\vee \otimes b} $ for $0 \leq a,b \leq n-1$.

We say that Hypothesis $\operatorname{H}(n,\rp,\rm,w)$ is satisfied if, for all such $\mathcal F$, \[H^j_c( \Prim_{n,\overline{\mathbb F}_q}, \mathcal F)=0\] for all $j> n+w$.  \end{hypothesis}

\begin{theo}\label{main} Let $n,\rp,\rm,w$ be natural numbers with $0 \leq w \leq n$ and $\mathbb F_q$ a finite field. Assume that Hypothesis $\operatorname{H}(n,\rp,\rm,w)$ is satisfied. Assume also that $n> 2 \max(\rp,\rm)+1$ and if $n=4$ or $5$ that the characteristic of $\mathbb F_q$ is not $2$. Let $C_{\rp,\rm} = (2 + \max(\rp,\rm))^{\max(\rp,\rm)+1}$.  Let $\alpha_1,\dots,\alpha_{\rp+\rm}$ be imaginary numbers. Let $\epsilon_\chi$ be the $\epsilon$-factor of $L(s,\chi)$.   Then

\begin{equation}\ \label{eq-main-left} \frac{1}{ q^{n}- q^{n-1} } \sum_{\chi \in S_{n,q}}  \epsilon_\chi^{-\rm} \prod_{i=1}^{\rp+\rm} L(1/2- \alpha_i, \chi) \end{equation} \begin{equation}  \label{eq-main-right} =    \sum_{\substack{  S \subseteq \{1,\dots,\rp+\rm\} \\ |S|=\rm}} \prod_{i \notin S} q^{ \alpha_i (n-1)} \sum_{ \substack{ f_1,\dots, f_{\rp+\rm} \in \mathbb F_q[T] \\ \textrm{monic} \\ \prod_{i \in S} f_i / \prod_{i\notin S} f_i \in T^{\mathbb Z} }}  \prod_{i\in S} |f_i|^{ -\frac{1}{2} +\alpha_i} \prod_{i \notin S} |f_i|^{ - \frac{1}{2} - \alpha_i} \end{equation} \begin{equation}\label{eq-main-error} + O \left(  q^{\frac{w-n}{2}}  C_{\rp,\rm}^n  n^{\rp+\rm} \right) . \end{equation}

\end{theo} 

%\begin{remark} Using the functional equation, we can rewrite \[ \prod_{i=1}^{\rp+\rm} L(1/2- \alpha_i, \chi)  = q^{ -(n-1)  \sum_{i=\rp+1}^{\rp+\rm} \alpha_i} \prod_{i=1}^{\rp} L(1/2-\alpha_i,\chi) \prod_{i=\rp+1}^{\rp+\rm} \overline{ L(1/2- \alpha_i, \chi)},\] so $\rp$ is the number of ordinary factors and $\rm$ is the number of complex conjugate factors.  \end{remark}

%as long as the $\epsilon$ factors are sufficiently oscillatory. The only subtle point is how to interpret the average of the $\epsilon$-factors. We don't want to take the average of $\epsilon^{ r-|S|}$ over all characters $\chi$ as that average will vanish. Instead, we do the CFKRS average over only primitive characters of conductor $n+1$, and use Hypothesis~\ref{hypo} to ensure the average has square-root cancellation, and thus can be ignored, when $r-|S|$ is not divisible by $m$.

\begin{remark}The parameters $n,\rp,\rm$ of Hypothesis~\ref{hypo} and Theorem \ref{main} bear a clear relationship to the moment  \eqref{eq-main-left} - indeed, $n$ determines the conductor of the characters while $\rp$ and $\rm$ determine the powers we raise the $L$-function and its $\epsilon$-factor to. The meaning of the parameter $w$ is less clear, and so we explain it here.
 
 In Hypothesis $\operatorname{H}(n,\rp,\rm,w)$, $w$ determines the width of the region where we do not assume that the cohomology groups vanish.  By Poincar\'{e} duality and Artin's affine theorem, $H^j_c( \Prim_n, \mathcal F) =0$ for $j<n$, so under Hypothesis $\operatorname{H}(n,\rp,\rm,w)$, the only possible nonvanishing cohomology groups occur when $j$ ranges from $n$ to $n+w$. In particular, the larger $w$ is, the weaker an assumption Hypothesis~\ref{hypo} is.
 
 In Theorem~\ref{main}, $w$ does not affect the moment \eqref{eq-main-left} nor the main term \eqref{eq-main-right}, but only the bound for the error term \eqref{eq-main-error}. In particular, the larger $w$ is, the larger, and thus weaker, the bound \eqref{eq-main-error}.
 
 Thus we can get by with a weaker geometric hypothesis, at the cost of a weaker analytic result. Depending on our purposes we can use the parameter $w$ in two ways - either proving a cohomology vanishing statement, finding the least value of $w$ for which it implies Hypothesis $\operatorname{H}(n,\rp,\rm,w)$, and deducing the corresponding bound, or determining a desired bound and finding the greatest value of $w$ for which Hypothesis $\operatorname{H}(n,\rp,\rm,w)$ is sufficient to prove it. \end{remark}
 
\begin{remark}\label{vandermonde-remark} It follows from Theorem~\ref{main} that the main term \eqref{eq-main-right} is finite. In other words the sum of meromorphic functions is in fact a holomorphic function on some neighborhood of the locus where $\alpha_1,\dots,\alpha_{\rp+\rm}$ are imaginary.  We can also show this more directly:

The sum \eqref{eq-main-right} is manifestly symmetric in the variables $\alpha_1,\dots,\alpha_{\rp+\rm}$, because each term is invariant under permuting the variables in $S$ and the variables outside $S$, and summing over the possibilities for $S$ makes it invariant under permuting all the variables. So if we multiply \eqref{eq-main-right} by the Vandermonde determinant $\prod_{ 1\leq i_1 < i_2 \leq \rp+\rm} (q^{\alpha_{i_1}} - q^{\alpha_{i_2}})$, it becomes antisymmetric in the variables. When we multiply each individual term by the Vandermonde determinant, or even its factor $\prod_{i_1 \in S, i_2 \notin S} (q^{\alpha_{i_1}} - q^{\alpha_{i_2}})$, they become holomorphic (Lemma \ref{Euler-bounds}) - in fact we can express them by a convergent Euler product). So the sum times the Vandermonde is holomorphic. But any antisymmetric holomorphic function of $q^{\alpha_i}$ vanishes whenever $q^{\alpha_{i_1}}= q^{\alpha_{i_2}}$ for any $i_1, i_2$ and thus is divisible by the Vandermonde determinant, so \eqref{eq-main-right} itself is holomorphic.\end{remark}

 Lemma~\ref{leading-term}, below, clarifies some of the properties of the main term \eqref{eq-main-right}. It shows that \eqref{eq-main-right} is a nonzero polynomial in $n$ if $\alpha_1 =\dots = \alpha_{\rp+\rm}= 0$ and is a (nonzero) quasiperiodic function of $n$ if $\alpha_1,\dots, \alpha_{\rp+\rm}$ are all distinct. In the polynomial case, the main term dominates the error term as long as $w < (1 - \log_{q} C_{\rp,\rm}-\epsilon ) n$ for any $\epsilon>0$, as then the error term decays exponentially with $n$. Similarly, in the quasiperiodic case, the main term will dominate the error term for most $n$ as long as $w < (1 - \log_{q} C_{\rp,\rm}-\epsilon ) n$ for any $\epsilon>0$. More information about the main term, including the calculation of the leading term of this polynomial, is contained in Lemma~\ref{leading-term}.

\begin{remark} 
\begin{enumerate}

\item We explain why \eqref{eq-main-right} is indeed the prediction of the CFKRS recipe~\cite[\S4]{CFKRS} (or~\cite[\S4.2]{AK}) for this family. This is obtained by the 5-step process (1) start with a product of shifted $L$-functions, (2) apply the ``approximate" functional equation to each term (in the function field case, an exact formula, following from polynomiality and the usual functional equation), (3) average the sign of the functional equations, (4) replace each summand by its expected value when averaged over the family, (5) extend the sums by removing limits of summation.

In step (1) we start with \eqref{eq-main-left}.
In step (2) we apply the ``approximate" functional equation \[ L(1/2-\alpha_i,\chi) = \sum_{\substack {f \in \mathbb F_q[T] \\ \textrm{monic} \\ \deg f \leq \frac{n-1}{2}  }} \chi(f) |f|^{-1/2 - \alpha_i}  + \epsilon_\chi q^{-\alpha_i (n-1) } \sum_{\substack {f \in \mathbb F_q[T] \\ \textrm{monic} \\ \deg f < \frac{n-1}{2}  }} \overline{\chi}(f) |f|^{-1/2 + \alpha_i} \]  to obtain that \eqref{eq-main-left} is

\[ \sum_{S \subseteq \{1,\dots, \rp+\rm\} } \epsilon_\chi^{\rp - |S| } \prod_{i \not \in S} q^{ -\alpha_i (n-1)} \sum_{ \substack{ f_1,\dots, f_{\rp+\rm} \in \mathbb F_q[T] \\ \textrm{monic} \\ \deg f_i \leq \frac{n-1 -1_{i\not \in S}}{2} }} \chi(\prod_{i \in S} f_i) \overline{\chi}( \prod_{i \not \in S} f_i)  \prod_{i \in S} |f_i|^{-1/2-\alpha_i} \prod_{i \not \in S} |f_i|^{-1/2+\alpha_i}.\]

In step (3) we remove the terms where $\rp\neq |S|$, as the average of the root number $\epsilon_\chi^{ \rp- |S|}$ cancels there.  (We show it cancels as part of Lemma \ref{m-estimate}.) %This gives

%\[  \sum_{\substack { S \subseteq \{1,\dots, \rp+\rm\} \\ |S| =r }} \prod_{i \not \in S} q^{ -\alpha_i (n-1)} \sum_{ \substack{ f_1,\dots, f_{\rp+\rm} \in \mathbb F_q[T] \\ \textrm{monic} \\ \deg f_i \leq \frac{n-1 -1_{i\not \in S}}{2} }} \chi(\prod_{i \in S} f_i) \overline{\chi}( \prod_{i \not \in S} f_i)  \prod_{i \in S} |f_i|^{-1/2-\alpha_i} \prod_{i \not \in S} |f_i|^{-1/2+\alpha_i}.\]

In step (4) we observe that for $n$ sufficiently large, the average over $\chi \in S_{n,q}$ of $\chi(\prod_{i \in S} f_i) \overline{\chi}( \prod_{i \not \in S} f_i) $ vanishes unless $\prod_{i \in S} f_i / \prod_{i \not \in S} f_ i\in T^{\mathbb Z}$. %This gives   

%\[  \sum_{\substack { S \subseteq \{1,\dots, \rp+\rm\} \\ |S| =r }} \prod_{i \not \in S} q^{ -\alpha_i (n-1)} \sum_{ \substack{ f_1,\dots, f_{\rp+\rm} \in \mathbb F_q[T] \\ \textrm{monic} \\ \prod_{i \in S} f_i/ \prod_{i \not \in S} f_i \in T^{\mathbb Z} \\ \deg f_i \leq \frac{n-1 -1_{i\not \in S}}{2} }}   \prod_{i \in S} |f_i|^{-1/2-\alpha_i} \prod_{i \not \in S} |f_i|^{-1/2+\alpha_i}.\]

In step (5) we extend the sums by removing the degree condition, getting 

\[  \sum_{\substack { S \subseteq \{1,\dots, \rp+\rm\} \\ |S| =\rp }} \prod_{i \not \in S} q^{- \alpha_i (n-1)} \sum_{ \substack{ f_1,\dots, f_{\rp+\rm} \in \mathbb F_q[T] \\ \textrm{monic} \\ \prod_{i \in S} f_i/ \prod_{i \not \in S} f_i \in T^{\mathbb Z}}}   \prod_{i \in S} |f_i|^{-1/2-\alpha_i} \prod_{i \not \in S} |f_i|^{-1/2+\alpha_i}.\]

\item We can write the error term in Theorem~\ref{main} as \[ O \left( \left(q^n\right)^{ -\frac{1}{2} + \frac{w}{n} + \frac{ (\max(\rp,\rm)+1) \log ( \max(\rp,\rm)+ 2)}{\log q}+\epsilon} \right).\]  The error term predicted by~\cite{CFKRS} is always the size of the family raised to the power $-\frac{1}{2} + \epsilon$. Our exponent approaches the predicted square-root cancellation as long as $\frac{w}{n} \to 0$ and $\frac{ (\max(\rp,\rm)+1) \log ( \max(\rp,\rm)+2)}{\log q} \to 0$. 

\end{enumerate}
\end{remark}

In fact, we are able to verify some nontrivial cases of Hypothesis~\ref{hypo}. using results from \cite{me}. More precisely, we see in Lemma~\ref{power-savings} that when $\mathbb F_q$ is a field of characteristic $p$, then Hypothesis $\operatorname{H}(n,\rp,1, n+1 - \frac{p-2\rp}{p\rp} n )$ is satisfied for any $n,\rp$. This gives the following unconditional estimate:

\begin{cor}[Corollary~\ref{combined-result}] Let $n,\rp$ be natural numbers and $\mathbb F_q$ a finite field of characteristic $p$. Assume also that $n> 2 \rp+1$ and if $n=4$ or $5$ that the characteristic of $\mathbb F_q$ is not $2$. Let $C_{\rp,1} = (2 +\rp)^{\rp+1}$.  Let $\alpha_1,\dots,\alpha_{\rp+1}$ be imaginary numbers. Let $\epsilon_\chi$ be the $\epsilon$-factor of $L(\chi)$.   Then

\[ \frac{1}{ (q^{n}- q^{n-1}) } \sum_{\chi \in S_{n,q}}  \epsilon_\chi^{-1} \prod_{i=1}^{\rp+1} L(1/2- \alpha_i, \chi)\] \[ =    \sum_{j=1}^{\rp+1} q^{-\alpha_j (n-1)} \left( \frac{1}{1- q^{-\frac{1}{2} - \alpha_j}  }\prod_{i \neq j}  \frac{ 1- q^{-1 + \alpha_i - \alpha_j}}{ (1-q^{-\frac{1}{2}+ \alpha_i}) (1-q^{\alpha_i-\alpha_j})}\right) + O \left(  \sqrt{q}  \left(  q^{ - \frac{p-2\rp}{2p\rp} } C_{\rp,1} \right) ^n  n^{r+1} \right) . \]

\end{cor}

As $p$ goes to $\infty$ with fixed $\rp$, this bound converges to a power savings of $1/2\rp$.

The main term in the case $\rp=2,\rm=2$ is also given by an explicit, though complicated, rational function in $q^{\alpha_i}$ and $q$, see Lemma~\ref{fourth-moment}.

\vspace{15pt}

We present a summary of some of the key ideas in the proof of Theorem~\ref{main}.

First note that one subtlety in Theorem~\ref{main} is that \eqref{eq-main-right} is a sum of terms that have poles at the points we are most interested in studying, and that the poles only disappear when we sum all the terms. This makes it tricky to try to prove the theorem by splitting up the terms, as this could introduce infinities. As noted in Remark~\ref{vandermonde-remark}, we can remove the poles by multiplying by a suitable Vandermonde determinant, and this helps to prove that the sum is holomorphic. The first step in our proof is a purely algebraic description of what happens when we multiply \eqref{eq-main-left} by the same Vandermonde determinant. We show that the coefficients of monomials in $q^{\alpha_1},\dots,q^{\alpha_{\rp+\rm}}$ in this product will be averages over $\chi$ of Schur functions in the zeroes of $L(s,\chi)$ corresponding to irreducible representations of $GL_{n-1}$ (Lemma~\ref{L-representation-identity}). Before multiplying by the Vandermonde determinant, the coefficients were typically characters of highly reducible representations, so the Vandermonde determinant significantly simplifies \eqref{eq-main-left} as well.

These coefficients will be crucial to proving Theorem~\ref{main}. Both \eqref{eq-main-left} and \eqref{eq-main-right} can be expressed as Laurent series in the variables $q^{\alpha_1},\dots,q^{\alpha_{\rp+\rm}}$, and they remain Laurent series when each is multiplied by the Vandermonde determinant. We will prove Theorem~\ref{main} by showing that the coefficient of each individual monomial $q^{ -\sum_i \alpha_i d_i}$ in \eqref{eq-main-left} times the Vandermonde is equal to the coefficient of the same monomial in \eqref{eq-main-right} times the Vandermonde, up to a controlled error. We think of this strategy as being analogous to, in classical moment calculations, breaking a sum over integers into dyadic intervals and handling them separately, or in function field moments calculation breaking a sum over polynomials into many sums over polynomials of fixed degree. %It may be helpful to think of the tuple $d_1,\dots, d_{\rp+\rm}$ that indexes a given monomial as coordinates in a $\rp+\rm$-dimensional hypercube, because the monomials appearing in \eqref{eq-main-left} have $(d_1,\dots,d_{\rp+\rm} )\in [0,n-1]^{\rp+\rm}$.

Typically in moment estimates, we break up the sum into distinct ranges, and there will some ranges in in which we can show that the off-diagonal terms cancel using only orthogonality of characters. The tuples $(d_1,\dots, d_{\rp+\rm})$ for which orthogonality of characters is enough are described by Lemma~\ref{cancellation-agreement}. Using this fact, we match some of the terms from \eqref{eq-main-left} with some of the terms to \eqref{eq-main-right}. We will then individually bound all the unmatched terms.

%If we view the set of tuples $(d_1,\dots, d_{\rp+\rm})$ as an $\rp+\rm$-dimensional hypercube, then for each set $S$ of $\{1,\dots, \rp+\rm\}$ we have a corresponding corner of the hypercube and a corresponding diagonal term, and Lemma~\ref{cancellation-agreement} says that the coefficients of monomials near that corner agree with that diagonal term. 
%
%In other words we are in a somewhat strange situation where we have a single coefficient in the product of \eqref{eq-main-left} with the Vandermonde which is supposed to be (approximately) equal to a sum of all the coefficients in the product of different terms on \eqref{eq-main-right} with the Vandermonde, while we know that it is equal to a single one of those coefficients. The simplest way to prove this would be to show that the other coefficients in \eqref{eq-main-right} are small.  This is indeed what we do, in the proof of Proposition~\ref{main1}, using estimates from Lemma~\ref{shifted-Euler-bounds}. 
%
%In addition, we must prove an approximate equality between coefficients in \eqref{eq-main-left} and \eqref{eq-main-right} outside the ranges where we can prove an exact equality. In this case, the simplest way to prove this is to show that both sides are very small.

We bound the unmatched terms in \eqref{eq-main-left} using our cohomology vanishing hypothesis, combined with the Grothendieck-Lefschetz fixed point formula and some Betti number estimates (Lemma~\ref{Betti-number-bound}), to prove that the average over $\chi$ of Schur functions in the zeroes of $L(s,\chi)$ corresponding to irreducible representations of $GL_{n-1}$ that do not appear as a summand of $L_{\univ}^{\otimes a}  \otimes L_{\univ}^{\vee \otimes b} $ for $0 \leq a,b \leq n-1$ is small.

We bound the unmatched terms in \eqref{eq-main-right} using estimates from Lemma~\ref{shifted-Euler-bounds}. These estimates are proved by expressing the diagonal term times the Vandermonde determinant as an Euler product, controlling each term of the Euler product, deducing a region of holomorphicity for the Euler product and an upper bound near the boundary of that region, then using a contour integral to show the coefficients decay as we get further from the corner. This is in contrast to the situation before multiplying by the Vandermonde determinant, where the coefficients of many different monomials in the diagonal term are large. 

\begin{remark} We present some remarks on the hypothesis, with the first two from an analytic perspective and the remainder from a geometric perspective.

\begin{enumerate}

\item  To obtain predictions for moments, instead of Hypothesis $\operatorname{H}(n,\rp,\rm,w)$, we could make a purely analytic conjecture of square-root cancellation in the trace of the cohomology (equivalently, the sum of the Schur polynomial associated to this representation, evaluated at the roots of the $L$-function, over all primitive Dirichlet characters) for representations outside this special set.

Such a hypothesis is essentially equivalent to a uniform version of the conjecture of \cite{CFKRS} for shifted moments, as we can extract these individual coefficients by a Fourier series after multiplying by the Vandermonde determinant. However, Hypothesis $\operatorname{H}(n,\rp,\rm,w)$ would not follow directly from this unless the cohomology groups were proven to be pure.

If made uniform in $\rp,\rm$, such a hypothesis would imply conjectures for ratios and tuple correlations - presumably matching the predictions of \cite{ratios}, and therefore \cite{applications-ratios}. On the other hand, while Hypothesis $\operatorname{H}(n,\rp,\rm,w)$ can be stated uniformly in $\rp,\rm$, it would not imply a good estimate on the error term in the degree aspect unless stronger Betti number bounds than those in \S2.3 were proven.

Despite these difficulties, we have stated Hypothesis $\operatorname{H}(n,\rp,\rm,w)$ in a geometric way to motivate it as a natural statement (we would not have come up with it if it weren't for geometry) and to suggest the potential of a geometric proof.

\item We can view the averages of Schur polynomials of the roots of the $L$-function that appear in the analytic version of the hypothesis (equivalently, the functions $F(V)$ discussed below in section 2) as being a distant analogue of the exponential sums considered in the circle method, as they are the averages of characters of irreducible representations of $GL_{n-1}$ over an arithmetically natural finite set, while the exponential sums in the circle method are the averages of characters of irreducible representations of $\mathbb Z$ over an arithmetically natural finite set. In this case, the averages of characters of irreducible representations appearing in $\operatorname{std}^{\otimes a} \otimes \operatorname{std}^{\vee \otimes b}$ for $0 \leq a,b \leq n-1$ are the analogue of the major arcs, which we can calculate reasonably explicitly. The averages of characters of other irreducible representations are the analogue of the minor arcs, which we hope to prove cancel.

\item  As part of our proof, we will implicitly calculate the trace of Frobenius on the cohomology of sheaves $\mathcal F$ which do appear as a summand of $L_{\univ}^{\otimes a}  \otimes L_{\univ}^{\vee \otimes b} $ for $0 \leq a,b\leq n-1$. So our hypothesis is a version of the usual heuristic that what we cannot calculate should cancel. Of course, such a heuristic may be overly optimistic. Instead, what is interesting here is that it is a very straightforward and geometrically natural heuristic.

\item The calculations of the traces for the sheaves which do appear as a summand of $L_{\univ}^{\otimes a}  \otimes L_{\univ}^{\vee \otimes b} $ for $0 \leq a,b\leq n-1$ are closely related to Katz's calculations in \cite[\S5]{SQRT}. (In the case $N=n<p$, the sheaf $\mathcal F$ defined in \cite[\S4]{SQRT} is the restriction of $L_{\univ}$ to a hyperplane section, and essentially the same calculations as in \cite[\S5]{SQRT} can be done in this setting.) So the failure of square-root cancellation he demonstrates does not cause a problem for us, as it occurs exactly in the cases where we do not assume square-root cancellation. In fact, we show that the non-square-root terms that he observes correspond exactly to the secondary terms predicted by \cite{CFKRS}.

\item Hypothesis $\operatorname{H}(n,\rp,\rm,n-1)$ is known for every $n,\rp,\rm$ by Poincar\'{e} duality in \'{e}tale cohomology.  This gives a bound in the $q$ aspect whose error term is $O(q^{-\frac{1}{2}})$. This implies that the main term of Theorem \ref{main} must match, to within $O(q^{-\frac{1}{2}})$, the main term obtained by applying Katz's equidistribution result \cite[Theorem 1.2]{WVQKR} and performing a matrix integral. Our method in this case is simply a (more complicated) variant of the proof of Deligne's equidistribution theorem \cite[Theorem 3.5.3]{weil-ii}, which Katz uses in his proof, combined with the calculation of the matrix integral.

\item It is possible that some very strong form of Hypothesis $\operatorname{H}(n,\rp,\rm,w)$ could be true. For instance, we could take the parameter $w$ to be uniform in $n,\rp,\rm $. This would be equivalent to replacing the first condition on $\mathcal F$ by the condition that it appears as a summand of the tensor product of some tensor power of $L_{\univ}$ with some tensor power of its dual. Conceivably the uniform constant could be as low as $w=2$. However, it is likely to be easier to prove weaker special cases first, which is why we have stated it flexibly using multiple parameters.
\end{enumerate}

\end{remark}

\begin{remark}

We present some remarks on possible generalizations. We first discuss families that are harmonic in the sense of \cite{SST}, and then geometric families.

\begin{enumerate}

\item We expect that these results can be generalized to at least some families with orthogonal and symplectic symmetry type. The simplest cases for our method are probably the families of Dirichlet characters studied by Katz in \cite{WVQRW}, where both orthogonal and symplectic examples are given.  One simply replaces the Vandermonde determinant with, for the $r$th moment in the orthogonal case, \begin{equation}\label{orthogonal-factor} \prod_{1\leq i_1 <i_2 \leq r} (q^{\alpha_{i_1} } - q^{\alpha_{i_2}}) (q^{\alpha_{i_1} }  q^{\alpha_{i_2}}- 1)\end{equation} or, for the $r$th moment in the symplectic case, \begin{equation}\label{symplectic-factor}\prod_{1 \leq i \leq r} (q^{2\alpha_i} -1) \prod_{1\leq i_1 <i_2 \leq r} (q^{\alpha_{i_1} } - q^{\alpha_{i_2}}) (q^{\alpha_{i_1} }  q^{\alpha_{i_2}}- 1).\end{equation} The hypothesis needed then has to do with the cohomology of sheaves generated from the universal sheaves constructed by Katz in that paper.

These formulas arise from the algebra of the orthogonal and symplectic group respectively, and thus which one to use should depend only on the symmetry type of the $L$-function. Specifically, they can be calculated by attempting to repeat the proof of Lemma \ref{representation-generating-identity} in the orthogonal or symplectic case. One starts with the orthogonal or symplectic Jacobi-Trudi identity, which relates the irreducible representations of the group to the determinant of a matrix whose entries are wedge powers of the standard representation \cite[(24.25), Corollary 24.35, Corollary 24.45]{FH}. Using this, it is straightforward to express a multivariable power series whose coefficients are irreducible representations of the group as the determinant of a fixed Vandermonde-like matrix times a multivariable power series whose entries are tensor products of wedge powers of the standard representation. This matrix determinant can be evaluated by the product formula \eqref{orthogonal-factor} or \eqref{symplectic-factor}.

Alternately, one can use the identities \cite[Lemma 4 on p. 249 and Lemma 5 on. p. 257]{BumpGamburd} which express the product of $L$-functions as a sum of Schur functions of the zeroes associated to irreducible representations of the appropriate symplectic or orthogonal group times Schur functions of the variables $q^{\alpha_i}$ associated to irreducible representations of the ``Howe dual" group $Sp_{2r}$ or $O_{2r}$ respectively. The appropriate replacement for the Vandermonde determinant is then the Weyl denominator for the characters of $Sp_{2r}$ or $O_{2r}$, as appropriate.

\item Similar results can be proven for moments of an $L$-function of a fixed Galois representation twisted by a varying Dirichlet character, again conditional on a cohomological hypothesis. However, the dependency on $n$ in the error term may be worse or even ineffective, as Betti number bounds are more difficult in this case. If the Galois representation is an Artin representation splitting over the function field of a curve of bounded degree and genus, it should be possible to make the dependence on $n$ an effective exponential.

\item For other harmonic families of Dirichlet characters, such as those of squarefree modulus, stating properly an analogous hypothesis seems to require the use of higher-dimensional sheaf convolution Tannakian categories, which have not yet been connected to equidistribution. If that geometric setup is handled, there should not be any major new difficulties. The fourth absolute moment for Dirichlet characters of prime modulus was studied in \cite{Tamam}.

\item For families of automorphic forms on higher-rank groups, the $q\to\infty$ equidistribution theory is not yet available, which is a precondition for our method. 

\item New difficulties present themselves in the family of all quadratic Dirichlet characters with squarefree moduli of a given degree. This family has attracted the most attention in the function field setting, beginning with \cite{HoffsteinRosen} and \cite{AK2} on the first moment. Recently, improved estimates for the first four moments were obtained in ~\cite{Ffirst,Ftwothree,fourth}. Improved estimates on the third moment were obtained in~\cite{third}, demonstrating the existence of a secondary term and thereby verifying a prediction from~\cite{DGH}.

The difficulties in applying our method to this case start with the fact that there is no range of short sums where the off-diagonal terms cancel completely. Thus, there is no set of irreducible representations close to the trivial representation in highest weight space whose contributions can be exactly computed. Furthermore, the existence of a secondary term in the cubic case suggests that even for representations very far from the trivial representation in highest weight space, the contribution does not necessarily exhibit square-root cancellation and the term does not vanish above the middle degree. However, neither of these difficulties seems insurmountable, and it is possible that the representation-theoretical and cohomological approach can separate the main term from the secondary terms and shed light, if only conjecturally, on each.

\item For general geometric families, the situation is likely similar to, but more complicated than, the situation for quadratic Dirichlet characters.

\end{enumerate}
\end{remark}

While writing this paper, the author was supported by Dr. Max R\"{o}ssler, the Walter Haefner Foundation and the ETH Z\"urich Foundation, and while finishing it, served as a Clay Research Fellow. The author would like to thank Emmanuel Kowalski, Corentin Perret-Gentil, and the two anonymous referees for helpful comments on drafts of this paper.

\section{Representation theory and algebraic geometry}

 For any $d \geq 0$, define \[\lambda_d (\chi) = q^{-d/2}  \sum_{ \substack { f \textrm{ monic} \\ \textrm{ degree } d}} \chi(f),\] so that $L(s,\chi) = \sum_{d=0}^{n-1} \lambda_d(\chi) q^{d(1/2 -s)}$. Let $\epsilon_\chi= \lambda_{n-1}(\chi)$ be the $\epsilon$-factor of $L(s,\chi)$, so that $\lambda_{n-1-d}(\chi) = \epsilon_\chi \overline{ \lambda_d(\chi)}$. By the Riemann hypothesis or more directly from the explicit formula for $\epsilon_\chi$ in terms of Gauss sums, we have $|\epsilon_\chi|=1$ for all $\chi$.

\begin{defi}\label{katz-bijection} We define a map from points of $\Prim_n(\mathbb F_q)$ to primitive characters of $\left( \mathbb F_q[x]/x^{n+1}\right)^\times$. In fact, recalling that $\Prim_n$ is $\mathbb A^n - \mathbb A^{n-1}$, we will define a map from $\mathbb A^n(\mathbb F_q)$ to characters of $\left( \mathbb F_q[x]/x^{n+1}\right)^\times$. We defer to \cite[\S2,\S3]{WVQKR} for the proof that this defines a bijection between even characters and points of $\mathbb A^n(\mathbb F_q)$, and that the primitive ones correspond to exactly the points that do not lie $\mathbb A^{n-1}$. 

Recall that for a natural number $l$, the length $l$ Witt vectors $W_{l}(\mathbb F_q)$ are a ring whose elements are $l$-tuples of elements in $\mathbb F_q$, with addition and multiplication defined by the Witt polynomials. For each $m \leq n$ prime to $l$ let $l(m,n)$ be the least natural number such that $p^{l(m,n)} m > n$ and fix an additive character $\psi_m :W_{l(m,n)}(\mathbb F_p) = \mathbb Z/p^{l(m,n)} \to \mathbb C^\times$.

We have the Artin-Hasse exponential power series \[ AH(x) = e^{ - \sum_{k =0}^{\infty} x^{p^k}/p^k}\] whose coefficients are $p$-adic integers. Given an element in $\left(\mathbb F_q[x]/x^{n+1} \right)^\times$, we can express it uniquely as  \begin{equation}\label{artin-hasse-expression} a_0 \prod_{ 1\leq m p^e \leq n, m\textrm{ prime to }p}  AH( a_{mp^e} x^{mp^e})^{1/m}\end{equation} for $a_0 \in \mathbb F_q^\times, a_1,\dots,a_n \in \mathbb F_q$. This is because $AH( a_{mp^e} X^{mp^e})^{1/m}= 1 - a_{mp^e} x^{mp^e}/m + \dots$ and so we can inductively choose each $a_{m p^e}$ to fix the coefficient of the corresponding power of $x$.

For a tuple $b_1,\dots,b_n$ in $\mathbb F_q$ defining a point of $\mathbb A^n(\mathbb F_q)$, the associated character of $\left( \mathbb F_q[x]/x^{n+1} \right)^\times$ is the one that sends \eqref{artin-hasse-expression} to \[ \sum_{ m \leq n, m \textrm { prime to } p, m\leq n} \psi_m \left( \tr_{ W_{l(m,n)}(\mathbb F_q)}^{ W_{l(m,n)}(\mathbb F_p)} \left( (a_m, a_{pm}, \dots, a_{p^{(l(m,n)-1)}m})\times (b_m, b_{pm}, \dots, b_{p^{(l(m,n)-1)}m})\right)\right) \] where the multiplication of tuples denoted by $\times$ is taken in the ring of Witt vectors.\end{defi}

We recall that $L_{\univ}$ was defined, using Definition \ref{katz-bijection}, in Definition \ref{L-univ}.

Let $m$ be the order of the geometric monodromy group of the the determinant of $L_{\univ}$. 

Let $\mu$ be $(-1)^{m (n-1)}$ times the (unique) eigenvalue of $\Frob_q$ on the $m$th power of the determinant of $L_{\univ}(1/2)$. The Tate twist, which has the effect of multiplying all eigenvalues of Frobenius on $L_{\univ}$ by $q^{-1/2}$, normalizes these eigenvalues to have absolute value $1$, so the eigenvalue on the $m$th power of the determinant will also have absolute value $1$, and thus $\mu$ will as well.

Let $R(GL_{n-1})$ be the representation ring of $GL_{n-1}$ over $\mathbb Z$.

We fix throughout an embedding $\iota: \overline{\mathbb Q}_\ell \to \mathbb C$. Using it, we will abuse notation and identify elements of $\overline{\mathbb Q}_\ell$ with their images in $\mathbb C$ under $\iota$.

Let $F$ be the unique additive group homomorphism: $R(GL_{n-1}) \to \mathbb C$ whose value on a representation $V$ is \[ F(V) = \sum_{j\in \mathbb Z}   (-1)^j   \tr(\Frob_q, H^j_c(\Prim_{n, \overline{\mathbb F}_q}, V (L_{\univ}(1/2)))).\]

\subsection{$L$-functions and irreducible representations}

In this subsection, we relate the moments of $L$-functions that will be our main object of the study to the functions $F(V)$ for irreducible representations $V$. It culminates in Lemma~\ref{L-representation-identity}, which expresses the moment of $L$-functions, times a Vandermonde determinant, as a sum of $F(V)$. To do this, we must first in Lemma \ref{lambda-F-relation} relate the $L$-function coefficients to $F(V)$, then in Lemma~\ref{representation-generating-identity} prove an identity in the representation ring that lets us reduce to irreducible representations.  The calculation of multiplicities in Lemma \ref{representation-multiplicity-identity} is proved by the same methods and will be useful later in conjunction with our Betti number bounds.

\begin{lemma}\label{lambda-F-relation}For any $\rp,\rm$, $d= (d_1,\dots,d_{\rp+\rm})$,

 \[ \sum_{\chi \in S_{n,q}} \epsilon_\chi^{-\rm}  \prod_{i=1}^{\rp+\rm}  \lambda_{d_i}(\chi) =(-1)^{\sum_{i=1}^{\rp+\rm} d_i} F\left( \det^{-\rm} \otimes \bigotimes_{i=1}^{\rp+\rm} \bigwedge^{d_i }  \right).\]

 Furthermore, $\epsilon_\chi^m = \mu$ for all primitive $\chi$. 
 
\end{lemma} 

\begin{proof}   First we observe that, by the Grothendieck-Lefschetz fixed point formula \cite[Sommes trig. (1.1.1)]{sga4h} \[ F(V) =  \sum_{\chi \in \Prim_n (\mathbb F_q) } \tr(\Frob_{q,\chi}, V (L_{\univ}(1/2) )). \] 

By construction, $\Prim_n(\mathbb F_q)$ is in bijection with $S_{n,q}$.  It is therefore sufficient to prove that \[\tr(\Frob_{q,\chi}, \wedge^d (L_{\univ}(1/2))) = (-1)^d \lambda_d(\chi).\] 

Because the trace of any matrix on $\wedge^d$ of the standard representation is $(-1)^d$ times the $d$th coefficient of the characteristic polynomial of that matrix, this follows from the fact \eqref{L-univ-function} that the characteristic polynomial of $\Frob_{q,\chi}$ acting on  $L_{\univ}$ is $L(s,\chi)$.

A special case is that $\epsilon_\chi =\lambda_{n-1}(\chi) =(-1)^{n-1}  \tr(\Frob_{q,\chi}, \det (L_{\univ}(1/2)))$. Thus
\[ \epsilon_\chi^m = (-1)^{ m(n-1)}  \tr(\Frob_{q,\chi}, \det^{\otimes m} (L_{\univ}(1/2))) = \mu \] by the definition of $\mu$.

%e determinant is a lisse sheaf of rank one, with a single Frobenius eigenvalue. By assumption, its $m$th power is a constant sheaf with Frobenius eigenvalue $(-1)^{(n-1) m } \mu$. Hence its Frobenius eigenvalue must be an $m$th root of $(-1)^{(n-1) m } \mu$, proving the last claim.

\end{proof}

For $0 \leq d_1 \leq \dots \leq d_{\rp+\rm} \leq n-1$, let $V_{d_1,,\dots, d_{\rp+\rm}}$ be the irreducible representation of $GL_{n-1}$ associated to the conjugate partition of $d_{\rp+\rm},\dots,1$, in other words the representation whose highest weight character is $\lambda_1^{\rp+\rm} \dots \lambda_{d_1}^{\rp+\rm} \lambda_{d_1+1}^{\rm+\rm-1} \dots \lambda_{d_2}^{\rp+\rm-1} \dots \lambda_{d_{\rp+\rm}}^1 $. Let \[V_{d_1,\dots,d_\rp | d_{\rp+1},\dots,d_{\rp+\rm}} = V_{d_1,\dots, d_{\rp+\rm}} \otimes \det^{-\rm},\] so its highest weight character is $\lambda_1^\rp\dots\lambda_{d_1}^\rp \lambda_{d_1+1}^{\rp-1} \dots \lambda_{d_2}^{\rp-1} \dots \lambda_{d_\rp}  \lambda_{d_{\rp+1} +1}^{-1} \dots \lambda_{n-1}^{-\rm}$.   

\begin{lemma}\label{representation-generating-identity} 
In the ring  $ R (GL_{n-1} ) [ q^{\alpha_1},\dots, q^{\alpha_{\rp+\rm}}]$ with formal variables $q^{\alpha_1}, \dots, q^{\alpha_{\rp+\rm}}$,  \begin{equation} \label{rgi-left}  \left( \sum_{\sigma \in S_{\rp+\rm}}  \sign(\sigma) \prod_{i=1}^{\rp+\rm} q^{ (\sigma(i)-1)\alpha_i  } \right) \sum_{0 \leq d_1,\dots, d_{\rp+\rm}  \leq n-1}   q^{ \sum_{i=1}^{\rp+\rm} d_i \alpha_i  }  (-1)^{ \sum_{i=1}^{\rp+\rm} d_i + {\rp+\rm\choose 2}  }\det^{-\rm} \otimes  \bigotimes_{i=1}^{\rp+\rm} \wedge^{d_i}   \end{equation} \begin{equation} \label{rgi-right}  =    \sum_{\sigma \in S_{\rp+\rm} }  \sum_{0 \leq d_1 \leq \dots \leq d_{\rp+\rm} \leq n-1}  \sign(\sigma) q^{ \sum_{i=1}^{\rp+\rm} (d_i+ i-1 ) \alpha_{\sigma(i)}  }  (-1)^{ \sum_{i=1}^{\rp+\rm}  d_i  }  V_{d_1,\dots,d_\rp | d_{\rm+1},\dots, d_{\rp+\rm} }.\end{equation}  \end{lemma}

\begin{proof} 
To check this, observe that \eqref{rgi-left} and \eqref{rgi-right} are antisymmetric in $\alpha_1,\dots,\alpha_{\rp+\rm}$. Hence it is sufficient to check that the coefficients of $q^{ \sum_{i=1}^{\rp+\rm} (d_i+ i-1 )\alpha_{i}  } $ in \eqref{rgi-left} and \eqref{rgi-right} agree for $ d_1\leq d_2 \leq \dots \leq d_{\rp+\rm}$.  Only the trivial permutation contributes to the coefficient of $q^{ \sum_{i=1}^{\rp+\rm} (d_i+ i-1 )\alpha_{i}  } $ in \eqref{rgi-right}, since the tuple $(d_1,\dots, d_{\rp+\rm}+\rp+\rm-1)$ is always in strictly increasing order, but applying any nontrivial permutation to that tuple will produce a tuple not in strictly increasing order. Thus the coefficient of $q^{ \sum_{i=1}^{\rp+\rm} (d_i+ i-1 )\alpha_{i}  } $ in \eqref{rgi-right} is $ (-1)^{ \sum_{i=1}^{\rp+\rm} d_i + {\rp+\rm \choose 2} } V_{d_1,\dots,d_\rp | d_{\rp+1},\dots, d_{\rp+\rm} }$.

In \eqref{rgi-left}, every permutation $\sigma$ contributes the amount \[\sign(\sigma)    (-1)^{\sum_{i=1}^{\rp+\rm} d_i + {\rp+\rm \choose 2}}\det^{-\rm} \otimes  \bigotimes_{i=1}^{\rp+\rm} \wedge^{ d_i + i -1 -( \sigma(i)-1)}   \] to this coefficient, so it suffices to check that
\[ V_{d_1,\dots,d_\rp | d_{\rp+1},\dots, d_{\rp+\rm} } =\left( \sum_{\sigma\in S_{\rp+\rm}} \sign(\sigma)\det^{-\rm} \otimes  \bigotimes_{i=1}^{\rp+\rm} \wedge^{ d_i + i-\sigma(i)}  \right)    .\]

Here we interpret wedge powers as vanishing if the power does not lie between $0$ and $n-1$. 

As $V_{d_1,\dots,d_\rp | d_{\rp+1},\dots, d_{\rp+\rm} }= \det^{-\rm} \otimes V_{d_1,\dots,d_{\rp+\rm}}$, it suffices to check that \begin{equation}\label{Jacobi-Trudi}  V_{d_1,\dots, d_{\rp+\rm} } = \sum_{\sigma \in S_{\rp+\rm}} \sign(\sigma) \bigotimes_{i=1}^{\rp+\rm} \wedge^{ d_i + i-\sigma(i)}    .\end{equation}

Observe that the right side of \eqref{Jacobi-Trudi} is the determinant of an $(\rp+\rm) \times (\rp+\rm)$ matrix whose $i,j$ entry is $\wedge^{ d_i + i-j}$. By the second Jacobi-Trudi identity for Schur functions \cite[Formula A6]{FH}, this determinant is equal to $V_{d_1,\dots,d_{\rp+\rm}}$ in the representation ring of $GL_{n-1}$. \end{proof}

\begin{lemma}\label{representation-multiplicity-identity} The multiplicity that $V_{d_1,\dots,d_\rp  | d_{\rp+1},\dots, d_{\rp+\rm} }$ appears in $\bigoplus_{ 0\leq e_1, \dots, e_{\rp+\rm} \leq n-1} \det^{-s} \otimes \bigotimes_{i=1}^{\rp+\rm} \wedge^{e_i}$  is  \[ \lim_{\alpha_1,\dots,\alpha_r \to 1}  \frac{ \sum_{\sigma \in S_{\rp+\rm}} \sign(\sigma) q^{( d_i +i-1)\alpha_\sigma(i)}}{ \prod_{1\leq i_1< i_2 \leq \rp+\rm} (q^{\alpha_{i_2}}- q^{\alpha_{i_1}})} .\]  \end{lemma}

\begin{proof} This is obtained from progressively similifying the formula of Lemma~\ref{representation-generating-identity}. First we make the substitution $q^{\alpha_i} \to -q^{\alpha_i}$, which removes some powers of $-1$.  Then we apply the linear map from $R(GL_{n-1})$ to $\mathbb Z$ that sends $V_{d_1,\dots,d_\rp  | d_{\rp+1},\dots, d_{\rp+\rm} }$ to $1$ and every other irreducible representation to $0$. This gives 
\[\left( \sum_{\sigma \in S_{\rp+\rm}}  \sign(\sigma) \prod_{i=1}^{\rp+\rm} q^{ (\sigma(i)-1)\alpha_i  } \right) \sum_{0 \leq d_1,\dots, d_{\rp+\rm}  \leq n-1}   q^{ \sum_{i=1}^{\rp+\rm} d_i \alpha_i  }  \operatorname{mult }\left( V_{d_1,\dots,d_\rp  | d_{\rp+1},\dots, d_{\rp+\rm} } , \det^{-\rm} \otimes  \bigotimes_{i=1}^{\rp+\rm} \wedge^{d_i} \right)  \] \[=   \sum_{\sigma \in S_{\rp+\rm} }  \sign(\sigma) q^{ \sum_{i=1}^{\rp+\rm} (d_i+ i-1 ) \alpha_{\sigma(i)}  }\] (where $\operatorname{mult}(V,W)$ is the multiplicity of the irreducible representation $V$ in the representation $W$).

We then divide both sides by $ \left( \sum_{\sigma \in S_{\rp+\rm}}  \sign(\sigma) \prod_{i=1}^{\rp+\rm} q^{ (\sigma(i)-1)\alpha_i  } \right) =  \prod_{1\leq i_1< i_2 \leq \rp+\rm} (q^{\alpha_{i_2}}- q^{\alpha_{i_1}}) $ and take the limit as $\alpha_1,\dots, \alpha_{r+s}$ go to $1$, obtaining
\[ \sum_{0 \leq d_1,\dots, d_{\rp+\rm}  \leq n-1}    \operatorname{mult }\left( V_{d_1,\dots,d_\rp  | d_{\rp+1},\dots, d_{\rp+\rm} } , \det^{-\rm} \otimes  \bigotimes_{i=1}^{\rp+\rm} \wedge^{d_i} \right) = \lim_{\alpha_1,\dots,\alpha_r \to 1}  \frac{ \sum_{\sigma \in S_{\rp+\rm}} \sign(\sigma) q^{( d_i +i-1)\alpha_\sigma(i)}}{ \prod_{1\leq i_1< i_2 \leq \rp+\rm} (q^{\alpha_{i_2}}- q^{\alpha_{i_1}})} \]
and finally observe that the sum of the multiplicity of $V_{d_1,\dots,d_\rp  | d_{\rp+1},\dots, d_{\rp+\rm}} $ in a sequence of representations is equal to the multiplicity of $V_{d_1,\dots,d_\rp  | d_{\rp+1},\dots, d_{\rp+\rm} }$ in the direct sum of these representations. \end{proof}

\begin{lemma}\label{L-representation-identity} We have the identity \[ \prod_{1 \leq i_1< i_2 \leq \rp+\rm} (q^{\alpha_{i_1}} - q^{\alpha_{i_2}}) \sum_{\chi \in S_{n,q} }\epsilon_\chi^{-s} \prod_{i=1}^{\rp+\rm} L(1/2- \alpha_i, \chi)  \]  \[ = \sum_{\sigma \in S_{\rp+\rm} } \sum_{0 \leq d_1 \leq \dots \leq d_{\rp+\rm} \leq n-1}  \sign(\sigma) q^{ \sum_{i=1}^{\rp+\rm} (d_i+ i-1 ) \alpha_{\sigma(i)}  }  (-1)^{ \sum_{i=1}^{\rp+\rm}  d_i    }  F\left(  V_{d_1,\dots,d_\rp | d_{\rp+1},\dots, d_{\rp+\rm } }\right ) \]  \end{lemma}

\begin{proof}

We have \[\sum_{\chi \in S_{n,q}  } \epsilon_\chi^{-\rm} \prod_{i=1}^{\rp+\rm} L(1/2- \alpha_i, \chi)     = \sum_{\chi \in S_{n,q}}  \sum_{0 \leq d_1,\dots, d_{\rp+\rm} \leq n-1}  q^{  \sum_{i=1}^{\rp+\rm}  d_i \alpha_i  }  (-1)^{ \sum_{i=1}^{\rp+\rm}  d_i  }  \epsilon_\chi^{-\rm}  \prod_{i=1}^{\rp+\rm}  \lambda_{d_i}(\chi)   .\]

Applying Lemma \ref{lambda-F-relation}, this is 
\[  \sum_{0 \leq d_1,\dots, d_{\rp+\rm} \leq n-1}  q^{ \sum_{i=1}^{\rp+\rm}  d_i \alpha_i  }  (-1)^{ \sum_{i=1}^{\rp+\rm}  d_i  }     F\left(  \det^{-\rm}  \otimes \bigotimes_{i=1}^{\rp+\rm}  \wedge^{d_i}   \right) .\]

Now multiply by the Vandermonde factor \[\prod_{1 \leq i_1< i_2 \leq \rp+\rm} (q^{\alpha_{i_1}} - q^{\alpha_{i_2}}) =(-1)^{ {\rp+\rm \choose 2}}  \sum_{\sigma \in S_{\rp+\rm}} \sign(\sigma) \prod_{i=1}^{\rp+\rm} q^{ \alpha_i  (\sigma(i)-1)}\] to obtain \[ \left( \sum_{\sigma \in S_{\rp+\rm}}\sign(\sigma) \prod_{i=1}^{\rp+\rm} q^{ \alpha_i  (\sigma(i)-1)} \right)  \sum_{0 \leq d_1,\dots, d_{\rp+\rm} \leq n-1}  q^{ \sum_{i=1}^{\rp+\rm}  d_i \alpha_i  }  (-1)^{ \sum_{i=1}^{\rp+\rm}  d_i + {\rp+\rm\choose 2}  }     F\left(  \det^{-\rm}  \otimes \bigotimes_{i=1}^{\rp+\rm}  \wedge^{d_i}   \right) \] which by Lemma~\ref{representation-generating-identity} is \[ = \sum_{\sigma \in S_{\rp+\rm} } \sum_{0 \leq d_1 \leq \dots \leq d_{\rp+\rm} \leq n-1}  \sign(\sigma) q^{ \sum_{i=1}^{\rp+\rm} (d_i+ i-1 ) \alpha_{\sigma(i)}  }  (-1)^{ \sum_{i=1}^{\rp+\rm}  d_i   }  F\left(  V_{d_1,\dots,d_\rp | d_{\rp+1},\dots, d_{\rp+\rm } }\right ). \] \end{proof}

\subsection{Auxiliary results}

In this subsection, we prove two lemmas that will be needed to concretely interpret the conditions on $\mathcal F$ in  Hypothesis $\operatorname{H}(n,\rp,\rm,w)$. We describe which $d_1,\dots, d_{\rp+\rm}$ have $V_{d_1,\dots,d_\rp | d_{\rp+1},\dots,d_{\rp+\rm}} ( L_{\univ}(1/2))$ appear as a summand of  $L_{\univ}^{\otimes a }\otimes L_{\univ}^{\vee \otimes b}$, which turns out to depend on the constant $m$, and then we prove a lemma that gives us some control on $m$.

\begin{lemma}\label{representation-theory-characterization}For  $0 \leq d_1 \leq \dots d_{\rp+\rm} \leq n-1$ and $0 \leq k \leq \rp+\rm,$  $V_{d_1,\dots,d_k | d_{k+1},\dots,d_{\rp+\rm}} $ appears as a summand $\operatorname{std}^{\otimes a} \otimes \operatorname{std}^{\vee \otimes b}$ for some $0 \leq a, b\leq n-1$ (possibly depending on $d_1,\dots, d_{\rp+\rm})$ if and only if $\sum_{i=1}^k d_i \leq n-1$ and $\sum_{i=k+1}^{\rp+\rm} (n-1-d_i) \leq n-1$ \end{lemma}

\begin{proof} For the only if direction, consider the element of $GL_{n-1}$ depending on a real parameter $\lambda>1$ whose eigenvalues are $\lambda$ with multiplicity $d_k$ and $1$  with multiplicity $n-1-d_k$. Its eigenvalue on the highest weight vector of $V_{d_1,\dots,d_k | d_{k+1},\dots,d_{\rp+\rm}} $ is $\lambda^{ \sum_{i=1}^k d_i}$. On the other hand, its eigenvalues on $\operatorname{std} ^{ \otimes a} \otimes \operatorname{std}^{\vee \otimes b}$ are at most $\lambda^a$ because its eigenvalues on $\operatorname{std}$ are at most $\lambda$ and its eigenvalues on $\operatorname{std}^\vee$ are at most $1$. Similarly, the element whose eigenvalues are $1$ with multiplicity $n-1-d_{k+1}$ and $\lambda^{-1}$ with multiplicity $d_k$ acts on the highest weight vector of $V_{d_1,\dots,d_k | d_{k+1},\dots,d_{\rp+\rm}} $ with eigenvalue $\lambda^{ \sum_{i=k+1}^{\rp+\rm}( n-1-d_i)}$, but its eigenvalues on $\operatorname{std} ^{ \otimes a} \otimes \operatorname{std}^{\vee \otimes b}$  are at most $\lambda^b$.

For the if direction, we observe that the highest weight of \begin{equation}\label{standard-power-construction} \left( \bigotimes_{i=1}^k \wedge^{d_i} \operatorname{std} \right) \otimes \left( \bigotimes_{i=k+1}^{\rp+\rm} \wedge^{n-1-d_i} \operatorname{std}^\vee\right)\end{equation} equals the highest weight of $V_{d_1,\dots,d_k | d_{k+1},\dots,d_{\rp+\rm}}$, and that \eqref{standard-power-construction} is a summand of $\operatorname{std}^{\otimes \sum_{i=1}^k d_i} \otimes \operatorname{std}^{\vee \otimes \sum_{i= k+1}^{\rp+\rm} (n-1-d_i)}$. \end{proof}

\begin{lemma}\label{easy-rep-characterization} Assume $n\geq 3$, and, if $n=3$, that the characteristic of $\mathbb F_q$ is not $2$ or $5$.

For $0 \leq d_1 \leq \dots d_{\rp+\rm} \leq n-1$,  the sheaf $V_{d_1,\dots,d_\rp | d_{\rp+1},\dots,d_{\rp+\rm}} ( L_{\univ}(1/2))$ appears as a geometric summand of  $L_{\univ}^{\otimes a }\otimes L_{\univ}^{\vee \otimes b}$ for $0 \leq a,b \leq n-1$ if and only if there is some $k$ such that $\sum_{i=1}^k d_i \leq n-1$, $\sum_{i=k+1}^{\rp+\rm} (n-1-d_i )\leq n-1$, and $k \equiv \rp \mod m$ \end{lemma}

\begin{proof} By~\cite[Theorem 7.1]{WVQKR}, under these assumptions on $n$, the monodromy group of $L_{\univ}$ contains $SL_{n-1}$. Thus two irreducible representations of $GL_{n-1}$ give isomorphic sheaves when composed with $L_{\univ}$ only if one is equal to the other twisted by an integer power of the determinant. Because $m$ is the order of the geometric monodromy group of the determinant, in fact they are isomorphic if and only if the integer is a multiple of $m$. The claim then follows from Lemma~\ref{representation-theory-characterization}.

\end{proof}

\begin{lemma}\label{m-estimate} The natural number $m$ from the beginning of this section is divisible by the largest power of $p$ that is greater than or equal to $\frac{n-1}{2}$. \end{lemma}

\begin{proof}Because $\mu = \epsilon_\chi^m$ for all $\chi$ and $|\epsilon_\chi|=1$, we have $|\mu|=1$. 

Fix  $\chi \in S_{n,q}$. %of $(\mathbb F_q[x]/x^{n+1})^\times$ trivial on $\mathbb F_q^\times$ and nontrivial on $1 + \mathbb F_q x^n$.
There is a unique nontrivial character $\psi_{\chi}: \mathbb F_q \to \mathbb C^\times$ with $ \chi(1 + ax^n) = \psi_{\chi}(-a)$ for all $a \in \mathbb F_q$.  We have \[\epsilon_\chi  = q^{ - \frac{n-1}{2}} \sum_{a_1,\dots,a_{n-1} \in \mathbb F_q} \chi( 1+ \sum_i a_i x^i)  =q^{ -\frac{n+1}{2}} \sum_{a_1,\dots,a_n \in \mathbb F_q} \chi\left( 1+ \sum_i a_i x^i\right) \psi_\chi (a_n) \] but for $\psi \neq \psi_\chi$,  \[ q^{ -\frac{n+1}{2}} \sum_{a_1,\dots,a_n \in \mathbb F_q} \chi\left( 1+ \sum_i a_i x^i\right) \psi(a_n) =0.\] Now fixing a nontrivial character $\psi$, there are $q^{n-1}$ characters $\chi \in S_{n,q}$ with $\psi_\chi=\psi$, so \begin{equation}\label{kloosterman-monodromy-evaluation} \sum_{\chi \in S_{n,q}}  \left( q^{ -\frac{n+1}{2}} \sum_{a_1,\dots,a_n \in \mathbb F_q} \chi( 1+ \sum_i a_i x^i) \psi(a_n)  \right)^m = \sum_{ \substack{ \chi \in S_{n,q}\\ \psi_\chi=\psi} }  \epsilon_\chi^m = \sum_{ \substack{ \chi \in S_{n,q}\\  \psi_\chi=\psi} }   \mu = q^{n-1} \mu\end{equation} On the other hand, we have 
\begin{equation}\label{kloosterman-cancellation-evaluation} \sum_{\chi \in S_{n,q}}  \left( q^{ -\frac{n+1}{2}} \sum_{a_1,\dots,a_n \in \mathbb F_q} \chi( 1+ \sum_i a_i x^i) \psi(a_n)  \right)^m \end{equation} \[  =  \sum_{ \chi: \left(\mathbb F_q[T]/T^{n+1}\right)^\times / \mathbb F_q^\times \to \mathbb C^\times  }  \left( q^{ -\frac{n+1}{2}} \sum_{a_1,\dots,a_n \in \mathbb F_q} \chi( 1+ \sum_i a_i x^i) \psi(a_n)  \right)^m \]
\[ =q^{- m \left( \frac{n+1}{2} \right)}  \sum_{\substack{ a_{i,j} \in \mathbb F_q, i=1,\dots,n, j=1,\dots m}}  \sum_{ \chi: \left(\mathbb F_q[T]/T^{n+1}\right)^\times / \mathbb F_q^\times \to \mathbb C^\times  }  \chi ( \prod_{j=1}^m (1+ \sum_{i=1}^n a_{i,j} x^i) )   \psi \left( \sum_{j=1}^m a_{n,j} \right)  \] 
\[=q^{- m \left( \frac{n+1}{2} \right)}   \sum_{\substack{ a_{i,j} \in \mathbb F_q, i=1,\dots,n, j=1,\dots m\\ \prod_{j=1}^m (1+ \sum_{i=1}^n a_{i,j} x^i) \equiv 1 \mod x^{n+1} }}   q^n  \psi \left( \sum_{j=1}^m a_{n,j} \right) \]
where we extend the sum from primitive $\chi$ to all $\chi$ because $\sum_{a_n \in \mathbb F_q} \chi(1+ \sum_i a_i x^i) \psi(a_n) =0$ for imprimitive $\chi$ and then expand the product and use orthogonality of characters.

Combining Equations \eqref{kloosterman-monodromy-evaluation} with \eqref{kloosterman-cancellation-evaluation}, we have  \begin{equation}\label{kloosterman-type-sum} \sum_{\substack{ a_{i,j} \in \mathbb F_q, i=1,\dots,n, j=1,\dots m\\ \prod_{j=1}^m (1+ \sum_{i=1}^n a_{i,j} x^i) \equiv 0 \mod x^{n+1} }}    \psi \left( \sum_{j=1}^m a_{n,j} \right) =   q^{ m \frac{n+1}{2} -1 } \mu \end{equation} where the power of $q$ shows up as $q^{ m \frac{n+1}{2} + (n-1) - n}$. If $m$ is not divisible by a large power of $p$, we will derive a contradiction from $|\mu|=1$ and the upper bound we will prove for the left side of \eqref{kloosterman-type-sum}.

The left side of \eqref{kloosterman-type-sum} is a Kloosterman-type sum. By standard stationary phase analysis, inductively for $i$ from $1$ to $\lfloor \frac{n-1}{2} \rfloor$, the sum over $a_{n-i,j}$ vanishes unless $a_{i,j_1} = a_{i,j_2}$ for all $j_1,j_2$. 

If $n$ is odd, after restricting to this subset the number of terms remaining in \eqref{kloosterman-type-sum} is $q^{ (m-1) ( \frac{n+1}{2} )}$ times the number of $a_1,\dots,a_{ \frac{n-1}{2} } $ satisfying $(1+ \sum_{i=1}^{ \frac{n-1}{2} } a_i x^i)^m \equiv 1  \mod x^{  \frac{n-1}{2} +1}$. The only way the total size is $q^{ m (\frac{n+1}{2})-1} $ is if the number of such $a_1,\dots,a_{\frac{n-1}{2}}$ is $q^{\frac{n-1}{2}}$, which only happens if the largest power of $p$ dividing $m$ is at least $\frac{n-1}{2}$.

For $n$ even, we observe that when $a_{1,j_1},\dots,a_{\frac{n}{2}-1,j_2}$ for all $i \leq \frac{n}{2}-1$ and all $j_1, j_2$, then the sum over $a_{\frac{n}{2},1},\dots, a_{\frac{n}{2},m}$ is a quadratic Gauss sum in $m-1$ variables, which is nondegenerate unless $p|m$, in which case it has one-dimensional degeneracy locus. This means \eqref{kloosterman-type-sum} is at most $q^{ (m-1) \frac{n+1}{2} + \frac{1}{2} } $ times the number of $a_1,\dots,a_{\frac{n}{2}-1} $ satisfying $(1+ \sum_{i=1}^{ \frac{n}{2}-1 } a_i x^i)^m \equiv 1  \mod x^{  \frac{n}{2} }$. The only way this can be at least $q^{ m (\frac{n+1}{2})-1} $ is if the number of such $a_1,\dots, a_{\frac{n-1}{2}}$ is at least $q^{\frac{n-1}{2}}$, which only happens if the largest power of $p$ dividing $m$ is at  least $\frac{n-1}{2}$.

 \end{proof}
 
 \subsection{Betti number bounds}\label{betti-bounds}
 
 In this subsection, we bound the dimension of the cohomology groups $H^j_c ( \Prim_{n, \overline{\mathbb F}_q} , V ( L_{\univ}(1/2)))$ that occur in the definition of $F(V)$. Combined with Hypothesis $\operatorname{H}(n,\rp,\rm,w)$, this will lead to bounds on $F(V)$. This will proceed by relating these cohomology groups to cohomology groups of varieties with constant coefficients, and then estimating those using Betti number bounds due to Katz. This relation is a geometric version of the standard arithmetic argument where, using the approximate functional equation and orthogonality of characters, a moment of $L$-functions of Dirichlet characters is reduced to the count (weighted by a smooth function) of tuples of natural numbers satisfying a congruence condition.
 
 We now give a more geometric interpretation of the construction in Definition~\ref{katz-bijection}  Let the scheme $\Witt_n$ be $\prod_{ m \geq 1 \textrm{ prime to }p, m \leq n } W_{l(m,n)}$ where $W_{l(m,n)} \cong \mathbb A^{l(m,n)}$ is the scheme parameterizing length $l(m,n)$ Witt vectors. This is a product of commutative unipotent group schemes and hence is a commutative unipotent group scheme itself, isomorphic to $\mathbb A^n$. Each $\mathbb F_q$-point corresponds to an even Dirichlet character $\mathbb F_q[x]/x^{n+1} \to \mathbb C^{\times}$ by Definition~\ref{katz-bijection}, and the multiplication in the group structure corresponds to multiplication of characters.

  Katz constructs $L_{\univ}$ as $R^1 pr_{2,!}\mathcal L_{\univ}$ for a certain lisse rank one sheaf $\mathcal L_{\univ} $ on $\mathbb A^1 \times \Prim_n $, with $pr_2$ the projection onto $\Prim_n$ \cite[\S4]{WVQKR}. However, Katz's definition of $\mathcal L_{\univ}$ works equally well to construct a lisse rank one sheaf on $\mathbb A^1 \times \Witt_n$. We will also refer to this sheaf as $\mathcal  L_{\univ}$ and the projection map $\mathbb A^1 \times \Witt_n \to \Witt_n$ as $pr_2$.

 For natural numbers $m_1, m_2$, let $Z_{n, m_1,m_2}$ be the subspace of $\mathbb A^{m_1} \times \mathbb A^{m_2}$ consisting of points $(a_{1},\dots,a_{m_1} ,b_{1},\dots,b_{m_2})$ such that $\prod_{i=1}^{m_1} ( 1 - a_ix )  \equiv \prod_{i=1}^{m_2} (1-b_i x) \mod x^{n+1}$.

\begin{lemma}\label{Fourier-comparison} For every $j \in \mathbb Z$, there is an $S_{m_1}\times S_{m_2}$ -equivariant isomorphism \[ H^j_c\left( \Witt_{n, \overline{\mathbb F}_q},  \left( R (pr_2)_! \mathcal L_{\univ} \right)^{ \otimes m_1} \otimes  \left( R (pr_2)_! \mathcal L_{\univ}^\vee \right)^{\otimes m_2} \right) = H^{j-2n}_c ( Z_{n,m_1,m_2, \overline{\mathbb F}_q},\mathbb Q_\ell (-n)).  \] 
  
  \end{lemma}
  
  \begin{proof} By applying the K\"{u}nneth formula \cite[Expos\'{e} XVII, Thm. 5.4.3]{sga4-3}, \[ H^j_c\left( \Witt_{n, \overline{\mathbb F}_q},  \left( R (pr_2)_! \mathcal L_{\univ} \right)^{ \otimes m_1} \otimes  \left( R (pr_2)_! \mathcal L_{\univ}^\vee \right)^{\otimes m_2} \right)  \] \[= H^j_c \left( \mathbb A^{m_1} \times \mathbb A^{m_2} \times \Witt_{n, \overline{\mathbb F}_q},   \bigotimes_{i=1}^{m_1}  \mathcal L_{\univ} (a_i, \omega)  \otimes \bigotimes_{i=1}^{m_2}  \mathcal L_{\univ}^\vee (b_i, \omega)\right) \] where $(a_1,\dots,a_{m_1},b_1,\dots,b_{m_2} )$  are coordinates on $  \mathbb A^{m_1} \times \mathbb A^{m_2} $ and $\omega$ is a coordinate on $\Witt_n$. 
  
 Let $pr_1$ be the projection $ \mathbb A^{m_1} \times \mathbb A^{m_2} \times \Witt_{n} \to \mathbb A^{m_1} \times \mathbb A^{m_2} $ and let $i: Z_{n, m_1,m_2} \to \mathbb A^{m_1} \times \mathbb A^{m_2}$ be the closed immersion. By applying the projection formula \cite[Expos\'{e} XVII, Prop. 5.2.9]{sga4-3} to $pr_1$ on the left side and $i$ on the right, it suffices to find an isomorphism \begin{equation}\label{basic-witt-isomorphism} R (pr_1)_! \left( \bigotimes_{i=1}^{m_1}  \mathcal L_{\univ} (a_i, \omega)  \otimes \bigotimes_{i=1}^{m_2}  \mathcal L_{\univ}^\vee (b_i, \omega)  \right) \cong i_! \mathbb Q_\ell[-2n] (-n) .\end{equation}
 
 To do this, we will first check that the stalk of $R (pr_1)_! \left( \bigotimes_{i=1}^{m_1}  \mathcal L_{\univ} (a_i, \omega)  \otimes \bigotimes_{i=1}^{m_2}  \mathcal L_{\univ}^\vee (b_i, \omega)  \right) $ vanishes outside the image of $i$. To do this, by proper base change \cite[Expos\'{e} XVII, Prop. 5.2.8]{sga4-3}, it suffices to check that the compactly supported cohomology of the fiber vanishes. It even suffices to check this for finite field-valued points, as the support is constructible.  Let $(a_1,\dots,a_{m_1}, b_1,\dots,b_{m_2} )\in \mathbb A^{m_1} \times \mathbb A^{m_2} (\mathbb F_q)$ be a point over a possibly larger finite field extension $\mathbb F_q$, and let $\mathcal L'$ be the fiber of $ \bigotimes_{i=1}^{m_1} \mathcal L_{\univ} (a_i, \omega)  \otimes \bigotimes_{i=1}^{m_2}  \mathcal L_{\univ}^\vee (b_i, \omega)$ over $(a_1,\dots,a_{m_1}, b_1,\dots,b_{m_2} )$, a sheaf lisse of rank one on $\Witt_n$. Over any finite field extension of $\mathbb F_q$, the trace function of $\mathcal L_{\univ} (a_i,\omega)$ is a $\Frob_q$-invariant character of $\Witt_n(\mathbb F_q)$ evaluated at $\omega$. Hence the trace function of $\mathcal L'$ is also a $\Frob_q$-invariant character. Thus the pullback of the trace function of $\mathcal L'$ under the Lang isogeny $\Witt_n \to \Witt_n, g \mapsto \Frob_q(g) g^{-1}$ is trivial, so by Chebotarev the pullback of $\mathcal L'$ under the Lang isogeny is trivial. So $\mathcal L'$ is a summand of the pushforward of the constant sheaf by the Lang isogeny of $\Witt_n$, and thus its cohomology is a summand of the cohomology of $\Witt_n$, which is $\mathbb Q_\ell(-n)$ in degree $2n$ because $\Witt_n \cong \mathbb A^n$.  Thus if $H^*_c( \Witt_n, \mathcal L')$ is nontrivial, it is equal to $\mathbb Q_\ell(-n)$, which implies that the sum of the trace function is $q^n$, so the character of $\Witt_n(\mathbb F_q)$ induced by $(a_{1},\dots,a_{m_1}, b_1,\dots,b_{m_2})$ is trivial, which contradicts the claim that $(a_1,\dots,a_{m_1},b_1,\dots,b_{m_2} ) \notin Z_{n,m_1,m_2}$. 
 
 So in fact $R (pr_1)_! \left( \bigotimes_{i=1}^{m_1}  \mathcal L_{\univ} (a_i, \omega)  \otimes \bigotimes_{i=1}^{m_2}  \mathcal L_{\univ}^\vee (b_i, \omega)  \right) $ is supported on the image of $i$. Restricting to the inverse image under $pr_1$ of the image of $i$, the trace function of $\bigotimes_{i=1}^{m_1} \mathcal L_{\univ} (a_i, \omega)  \otimes \bigotimes_{i=1}^{m_2}  \mathcal L_{\univ}^\vee (b_i, \omega)$ is the constant function $1$, because the trace function of $\mathcal L_{\univ} (a_i, \omega) $ is by construction the evaluation of the character corresponding to $\omega$ at $( 1 - a_ix )$, so the trace function of the tensor product is the product of the values of this character at $\prod_{i=1}^{m_1} ( 1 - a_ix )  / \prod_{i=1}^{m_2} (1-b_i x)$, which is $1$ as that element is the identity by the definition of $i$.
 
Because its trace function is $1$, $\bigotimes_{i=1}^{m_1} \mathcal L_{\univ} (a_i, \omega)  \otimes \bigotimes_{i=1}^{m_2}  \mathcal L_{\univ}^\vee (b_i, \omega) \cong \mathbb Q_\ell$, giving an isomorphism \[ i^* R (pr_1)_! \left(\bigotimes_{i=1}^{m_1} \mathcal L_{\univ} (a_i, \omega)  \otimes \bigotimes_{i=1}^{m_2}  \mathcal L_{\univ}^\vee (b_i, \omega) \cong \mathbb Q_\ell \right) \cong i^* R (pr_1)_! \mathbb Q_\ell \cong i^* \mathbb Q_\ell [-2n](n)\] and thus by the support condition an isomorphism \eqref{basic-witt-isomorphism}, as desired. 
 
 \end{proof}
 
\begin{lemma}\label{witt-excision} For $0 \leq d_1,\dots,d_{\rp+\rm} \leq n-1$, there is a long exact sequence of complexes of vector spaces  \[ H^*_c \left(\Prim_{n, \overline{\mathbb F}_q},  \bigotimes_{i=1}^{\rp}  \wedge^{d_i} (L_{\univ} ) [-d_i] \otimes \bigotimes_{i=\rp+1}^{\rp+\rm} \wedge^{d_i}(L_{\univ} ^\vee)[-d_i] \right) \] 

\[\to H^*_c \left(\Witt_{n,\overline{\mathbb F}_q}, \left( R (pr_2)_! \mathcal L_{\univ} \right)^{ \otimes \sum_{i=1}^\rp d_i} \otimes  \left( R (pr_2)_! \mathcal L_{\univ}^\vee \right)^{\otimes \sum_{i=\rp+1}^{\rp+\rm} d_i } \right)^{ S_{d_1} \times \dots \times S_{d_{\rp+\rm} }}\]

\[ \to H^*_c \left(\Witt_{n-1,\overline{\mathbb F}_q}, \left( R (pr_2)_! \mathcal L_{\univ} \right)^{ \otimes \sum_{i=1}^\rp d_i} \otimes  \left( R (pr_2)_! \mathcal L_{\univ}^\vee \right)^{\otimes \sum_{i=\rp+1}^{\rp+\rm} d_i } \right)^{ S_{d_1} \times \dots \times S_{d_{\rp+\rm} }}\] where $ S_{d_1} \times \dots \times S_{d_{\rp+\rm} } \subseteq S_{\sum_{i=1}^r d_i} \times S_{\sum_{i=\rp+1}^{\rp+\rm} d_i}$ in the obvious way. \end{lemma}

\begin{proof} In view of the excision long exact sequence \cite[Expos\'{e} XVII, Eq (5.1.16.2)]{sga4-3}, it suffices to find an isomorphism (over $\Prim_n$) \[\hspace*{-1cm}\bigotimes_{i=1}^{\rp}  \wedge^{d_i} (L_{\univ}) [-d_i] \otimes \bigotimes_{i=\rp+1}^{\rp+\rm} \wedge^{d_i}(L_{\univ}^\vee)[-d_i] \cong  \left(  \left( R (pr_2)_! \mathcal L_{\univ} \right)^{ \otimes \sum_{i=1}^\rp d_i} \otimes  \left( R (pr_2)_! \mathcal L_{\univ}^\vee \right)^{\otimes \sum_{i=\rp+1}^{\rp+\rm} d_i } \right)^{ S_{d_1} \times \dots \times S_{d_{\rp+\rm}}}.\]

To do this, observe that by definition \[ L_{\univ} = R^1 (pr_2)_! \mathcal L_{\univ}\] and that \[L_{\univ}^\vee = R^1(pr_2)_! \mathcal L_{\univ}^\vee\] because they are lisse, irreducible, and have the same trace functions up to scaling. Note too that $\mathcal L_{\univ}$ and its dual have no higher and lower cohomology in the fibers of $pr_1$ over $\Prim_n$, so that \[L_{\univ}[-1] = R (pr_1)_! \mathcal L_{\univ}, L_{\univ}^\vee[-1]= R (pr_2)_! \mathcal L_{\univ}^\vee .\]  Because the tensor product of complexes is anticommutative in odd degrees, we have \[ (\wedge^d L_{\univ}) [d] = \Sym^d (L_{\univ}[1]) = \left( \left( L_{\univ}[1] \right)^{\otimes d} \right)^{S_d} =  \left( \left( R (pr_2)_! \mathcal L_{\univ} \right)^{\otimes d} \right)^{S_d} \] and similarly for $\mathcal L_{\univ}^\vee$.  

Tensoring these equalities for $d_i$ from $1$ to $\rp$, and the dual equalities for $d_i$ from $\rp+1$ to $\rp+\rm$, we have the desired isomorphism. \end{proof}

\begin{lemma}\label{Betti-number-bound}For $0 \leq d_1,\dots,d_{\rp+\rm} \leq n-1$, we have the Betti number bound \begin{equation}\label{eq-betti-bound} \sum_j \dim H^j_c \left(\Prim_{n, \overline{\mathbb F}_q},  \bigotimes_{i=1}^{\rp}  \wedge^{d_i} (L)  \otimes \bigotimes_{i=\rp+1}^{\rp+\rm} \wedge^{d_i}(L^\vee)\right)\leq 4 (2+ \max(\rp,\rm) )^{ n + \sum_{i=1}^{\rp+\rm} d_i }  .\end{equation}

\end{lemma}

\begin{proof}  We apply the exact sequence of Lemma~\ref{witt-excision} and then evaluate each term using Lemma \ref{Fourier-comparison}. Because of this, it suffices to bound  \[ \sum_j \dim \left( H^j_c\left( Z_{n,\sum_{i=1}^\rp d_i, \sum_{i=\rp+1}^{\rp+\rm} d_i, \overline{\mathbb F}_q}, \mathbb Q_\ell \right)\right)^{ S_{d_1} \times \dots \times S_{d_{\rp+\rm} }} \] and \[ \sum_j \dim \left( H^j_c\left( Z_{n-1,\sum_{i=1}^\rp d_i, \sum_{i=\rp+1}^{\rp+\rm} d_i, \overline{\mathbb F}_q}, \mathbb Q_\ell \right)\right)^{ S_{d_1} \times \dots \times S_{d_{\rp+\rm} }} \] separately. We have \[  \left( H^j_c\left( Z_{n,\sum_{i=1}^\rp d_i, \sum_{i=\rp+1}^{\rp+\rm} d_i, \overline{\mathbb F}_q}, \mathbb Q_\ell \right)\right)^{ S_{d_1} \times \dots \times S_{d_{\rp+\rm} }}  =   H^j_c\left( Z_{n,\sum_{i=1}^\rp d_i, \sum_{i=\rp+1}^{\rp+\rm} d_i, \overline{\mathbb F}_q} / (  S_{d_1} \times \dots \times S_{d_{\rp+\rm} } ) , \mathbb Q_\ell\right). \]

Because we can take the coordinates of $\left( \mathbb A^d\right)^{S_d}$ to be the coefficients of the polynomial $\prod_{i=1}^T (T-a_i)$ for $a_1,\dots,a_d$ the coordinates of $\mathbb A^d$, we can view $ Z_{n,\sum_{i=1}^\rp d_i, \sum_{i=r+1}^{\rp+\rm} d_i, \overline{\mathbb F}_q} / ( S_{d_1} \times \dots \times S_{d_{\rp+\rm} })$ as the moduli space of tuples of monic polynomials $f_1,\dots,f_{\rp+\rm}$, with $f_i$ of degree $d_i$, such that the leading $n+1$ coefficients of $\prod_{i=1}^\rp f_i$ and $\prod_{i=\rp+1}^{\rp+\rm}f_i$ agree.  The equality of the leading coefficient is trivial, while the equality of the remaining $n$ coefficients is a system of $n$ polynomial equations of degrees $\max(\rp,\rm)$. So this is the solution set of a system of $n$ equations, of degree at most $\max(\rp,\rm)$, in $\sum_{i=1}^{\rp+\rm} d_i$ variables. By \cite[Theorem 12]{katzbetti}, \[\sum_j \dim_i H^j_c \left( Z_{n,\sum_{i=1}^\rp d_i, \sum_{i=\rp+1}^{\rp+\rm} d_i, \overline{\mathbb F}_q}/ ( S_{d_1} \times \dots \times S_{d_{\rp+\rm} } ), \mathbb Q_\ell\right) \leq 3  (2+ \max(\rp,\rm) )^{ n + \sum_{i=1}^{\rp+\rm} d_i }.\]

For $Z_{n-1}$, the same argument gives a Betti number bound of \[3  (2+ \max(\rp,\rm) )^{ n-1 + \sum_{i=1}^{\rp+\rm} d_i }\leq   (2+ \max(\rp,\rm) )^{ n + \sum_{i=1}^{\rp+\rm} d_i }\] as $2 + \max(\rp,\rm) \geq 3$. 

Summing the bounds for $Z_n$ and $Z_{n-1}$, we get \eqref{eq-betti-bound}.  \end{proof}

\section{Analysis of the Main Term}

%For polynomials $g_1, g_2$ in a variable $T$, say $g_1=_T g_2$ if $ \frac{g_1}{T^{\deg g_1}} = \frac{g_2}{ T^{\deg g_2}}$.

For $S \subseteq \{1, \dots, \rp+\rm \}$, let  \[M_S (\alpha_1,\dots,\alpha_{\rp+\rm}) =  \prod_{1 \leq i_1< i_2 \leq \rp+\rm} (q^{\alpha_{i_1}} - q^{\alpha_{i_2}}) \sum_{ \substack{ f_1,\dots, f_{\rp+\rm} \in \mathbb F_q[T] \\ \textrm{monic} \\ \prod_{i \in S} f_i  /  \prod_{i\notin S} f_i \in T^{\mathbb Z} } }  \prod_{i \in S} q^{ (-1/2 + \alpha_i)  \deg f_i }  \prod_{i\notin S} q^{( -1/2 - \alpha_i ) \deg f_i  }. \] Note that this is independent of $n$.

The main goal of this section is to estimate the coefficients of $M_S$, viewed as a Laurent series in the $q^{\alpha_i}$. In the first two lemmas we will establish some basic properties that will be useful later, in particular giving conditions for the coefficients of this series to be nonvanishing. In the next three lemmas, we will prove a bound on the coefficients using a contour integral argument. In Lemma \ref{cancellation-agreement}, we will relate the coefficients of $M_S$ to the coefficients of \eqref{eq-main-right}.

\begin{lemma}\label{Euler-support} Let $d_1,\dots,d_{\rp+\rm}$ be integers. The coefficient of $\prod_i q^{\alpha_i d_i}$  in \[  \sum_{ \substack{ f_1,\dots, f_{\rp+\rm} \in \mathbb F_q[T] \\ \textrm{monic} \\ \prod_{i \in S} f_i  /  \prod_{i\notin S} f_i \in T^{\mathbb Z}} }  \prod_{i \in S} q^{ (-1/2 + \alpha_i)  \deg f_i }  \prod_{i\notin S} q^{( -1/2 - \alpha_i ) \deg f_i  } \]
vanishes unless $d_i \geq 0$ for $i \in S$ and $d_i \leq 0$ for $i \notin S$.  Furthermore, the coefficient of $\prod_i q^{\alpha_i d_i}$ in this expression is symmetric in the variables $d_i$ for $i \in S$ and also in the $d_i$ for $i \notin S$. \end{lemma}

\begin{proof} The vanishing is because, in each term of the sum, $q^{\alpha_i}$ appears only in nonnegative powers if $i \in S$ and in nonpositive powers in $i \in S$.  The symmetry is because the definition is symmetric in the $\alpha_i$ for $i \in S$ and symmetric in the $\alpha_i$ for $i \notin S$, by permuting the corresponding $f_i$s.
\end{proof}

\begin{lemma}\label{M-support}  $M_S$ is antisymmetric in the $\alpha_i$ variables for $i \in S$, and also in the $\alpha_i$ for $i\notin S$.  Expressed as a power series, the coefficient of $q^{ \sum_i \alpha_i d_i}$ is nonzero only if there exists a permutation $\sigma \in S_{\rp+\rm}$ such that $\sigma(\{1,\dots, |S|\}) =S$, $ d_{\sigma(i)}  \geq i-1$ for all $i$ from $1$ to $|S|$, and $d_{\sigma(i)} \leq i-1$ for all $i$ from $|S|+1$ to $\rp+\rm$. \end{lemma}

\begin{proof}  These follow from Lemma~\ref{Euler-support} once we adjust for the Vandermonde factor. The first claim follows because the Vandermonde is antisymmetric, and it is multiplied by a symmetric term, making the product antisymmetric. The support conditions follow from the support statement in Lemma~\ref{Euler-support} combined with the fact that the Vandermonde determinant is a sum of terms $q^{ \sum_i \alpha_i d_i}$ where $d_1,\dots,d_{\rp+\rm}$, is a permutation of $0,\dots, \rp+\rm-1$.

\end{proof}

We next prove three lemmas that bound the coefficients of $M_S$. Lemma~\ref{general-contour-integral-bound} will give a general means for converting bounds for the values of a power series into bounds for that power series, based on a contour integral. Lemma~\ref{Euler-bounds} will use this to bound the coefficients of a simplified version of $M_S$, where we include only some of the coefficients in the Vandermonde factor.  Lemma~\ref{shifted-Euler-bounds} will apply this to bound $M_S$. 

\begin{lemma}\label{general-contour-integral-bound} Let $F( \alpha_1,\dots, \alpha_{\rp+\rm})$ be a power series in $q^{\alpha_i} $ ($ i \in S$) and $q^{- \alpha_i} $ ( $i \not \in S$) with complex coefficients which converges for any $(\alpha_1,\dots, \alpha_{\rp+\rm})$ such that $\operatorname{Re} \alpha_i$ is sufficiently small for $i$ in $S$ and sufficiently large for $i$ not in $S$. 

Let $c_1,\dots, c_{\rp+\rm}$ be real numbers such that  $F( \alpha_1,\dots, \alpha_{\rp+\rm})$ extends to a holomorphic function on the set of tuples $(\alpha_1,\dots, \alpha_{\rp+\rm})$ such that  $\operatorname{Re} (\alpha_i)< c_i$ if $i \in S$ and $ \operatorname{Re} (\alpha_i)>  c_i$ if $i \notin S$. Suppose further that, for any $\epsilon>0$, whenever $\operatorname{Re} (\alpha_i)\leq c_i- \epsilon$ if $i \in S$ and $ \operatorname{Re} (\alpha_i)\geq c_i+ \epsilon $ if $i \not \in S$, this holomorphic function is $ \left( \frac{1}{ 1 - q^{\epsilon}} \right)^{O(1)}$. 
 
Then the coefficient of $q^{ \sum_i d_i \alpha_i}$ in $F( \alpha_1,\dots, \alpha_{\rp+\rm})$ is \[ O\left( \left( 1 + \sum_{i \in S} d_i - \sum_{i\not \in S} d_i \right)^{ O(1)} q^{ - \sum_{i=1}^{\rp+\rm} c_i  d_i } \right) .\] \end{lemma}

\begin{proof} Let $\delta_i= -1$ if $i\in S$ and $1$ if $i\notin S$.
By the Cauchy integral formula, the coefficient of $q^{ \sum_i d_i \alpha_i}$ in $F( \alpha_1,\dots, \alpha_{\rp+\rm})$  is  \begin{equation}\label{eq-contour-integral} \frac{1}{ (2\pi)^{\rp+\rm} } \int_{ t_1,\dots, t_{\rp+\rm} \in [0,2\pi]}  \frac{  F ( c_1 + \delta_i \epsilon + i t_1, \dots, c_{\rp+\rm} + \delta_i \epsilon + i t_{\rp+\rm} ) }{q^{(c_1 + \delta_i \epsilon + i t_1)d_1 + \dots + (c_{\rp+\rm}  + \delta_i \epsilon + i t_{\rp+\rm} ) d_{\rp+\rm} }      } dt_1\dots d t_{\rp+\rm} \end{equation} where we have extended $F$ as a holomorphic function from its region of absolute convergence. Then by our assumed bound, \eqref{eq-contour-integral} is \[ \frac{1}{ (2\pi)^{\rp+\rm} } \int_{ t_1,\dots, t_{\rp+\rm} \in [0,2\pi]}  \frac{  O\left(  \left( \frac{1}{ 1 - q^{\epsilon}} \right)^{O(1)} \right)   }{q^{ \sum_{i=1}^{\rp+\rm} c_i d_i  -  \sum_{i \in S} \epsilon d_i +  \sum_{i \not \in S} \epsilon d_i  } } dt_1\dots d t_{\rp+\rm} = \frac{  O\left(  \left( \frac{1}{ 1 - q^{\epsilon}} \right)^{O(1)} \right)   }{q^{ \sum_{i=1}^{\rp+\rm} c_i d_i  -  \sum_{i \in S} \epsilon d_i +  \sum_{i \not \in S} \epsilon d_i  } } .\] Now let \[ \epsilon = \frac{1}{ (  1+ \sum_{i \in S} d_i - \sum_{i \notin S} d_i )\log q } .\] Then our bound for \eqref{eq-contour-integral} specializes to
\[  \frac{  O \left(   \left( 1+ \sum_{i \in S} d_i - \sum_{i \notin S} d_i \right)^{O(1) } \right)} {  q^{ \sum_{i=1}^{\rp+\rm} c_i d_i - O(1) }}  = O\left( \left( 1 + \sum_{i \in S} d_i - \sum_{i\not \in S} d_i \right)^{ O(1)} q^{ - \sum_{i=1}^{\rp+\rm} c_i  d_i } \right).\] \end{proof}

To simplify $M_S$, observe that \begin{equation}\label{MS-observation-1} \prod_{1 \leq i_1< i_2 \leq \rp+\rm} (q^{\alpha_{i_1}} - q^{\alpha_{i_2}}) =\pm  \left(\prod_{i_1 \in S} \prod_{i_2 \notin S}  (q^{\alpha_{i_2}} - q^{ \alpha_{i_1}} ) \right) \left( \prod_{\substack{  i_1, i_2\in S \\ i_1 < i_2 } } (q^{\alpha_{i_1}} - q^{\alpha_{i_2}}) \right)  \left( \prod_{ \substack{  i_1, i_2\notin S \\ i_1 < i_2 } } (q^{\alpha_{i_1}} - q^{\alpha_{i_2}}) \right)  \end{equation} and that we can further write \begin{equation}\label{MS-observation-2} \prod_{i_1 \in S} \prod_{i_2 \notin S}  (q^{\alpha_{i_2}} - q^{ \alpha_{i_1} }) =  \left( \prod_{i \not \in S} q^{ \rp \alpha_i } \right) \left(\prod_{i_1 \in S} \prod_{i_2 \notin S}  (1  - q^{ \alpha_{i_1} - \alpha_{i_2}})\right) . \end{equation} We will only use the factor $\prod_{i_1 \in S} \prod_{i_2 \notin S}  (1  - q^{ \alpha_{i_1} - \alpha_{i_2}})$, which can be expressed nicely as an Euler product, in Lemma \ref{Euler-bounds}, and will add the rest of the formula in Lemma~\ref{shifted-Euler-bounds}.

\begin{lemma}\label{Euler-bounds}  Let $d_1,\dots,d_{\rp+\rm}$ be integers with $\sum_{i\in S} d_i - \sum_{i \notin S} d_i \geq 0$. The coefficient of $\prod_i q^{\alpha_i d_i}$  in \begin{equation}\label{Euler-bounds-expression} \left( \prod_{i_1 \in S} \prod_{i_2 \notin S}\left(1 - q^{\alpha_{i_1} - \alpha_{i_2}}\right)\right) \sum_{ \substack{ f_1,\dots, f_{\rp+\rm} \in \mathbb F_q[T] \\ \textrm{monic} \\ \prod_{i \in S} f_i  /  \prod_{i\notin S} f_i \in T^{\mathbb Z}} }  \prod_{i \in S} q^{ (-1/2 + \alpha_i)  \deg f_i }  \prod_{i\notin S} q^{( -1/2 - \alpha_i ) \deg f_i  } \end{equation} is \[  O \left(   \left(1 + \sum_{i \in S} d_i - \sum_{i \notin S} d_i \right)^{O(1)}\min \left( q^{ \frac{ - \max_{i \in S} d_i + \min_{i \notin S} d_i}{2} },  q^{ \frac{ - \sum_{i  \in S}  d_i }{2 }}, q^{ \frac{  \sum_{i \notin S}  d_i }{2 }}, q^{- \frac{ \left| \sum_{i \in S}  d_i + \sum_{i \notin S} d_i \right|}{2}  } \right) \right)  .\]  
%\[  \left(1 + \sum_{i \in S} d_i - \sum_{i \notin S} d_i \right)^{O(1)} q^{ \frac{ - \sum_{i \in S}  d_i }{2 }} \]
%
%\[  \left(1 + \sum_{i \in S} d_i - \sum_{i \notin S} d_i \right)^{O(1)} q^{ \frac{ - \sum_{i \notin S}  d_i }{2 }} \]
%
%\[  \left(1 + \sum_{i \in S} d_i - \sum_{i \notin S} d_i \right)^{O(1)}q^{- \frac{ \left| \sum_{i \in S}  d_i + \sum_{i \notin S} d_i \right|}{2} } \]  
%
\end{lemma}

Note that, by Lemma~\ref{Euler-support}, this coefficient vanishes unless $\sum_{i\in S} d_i - \sum_{i \notin S} d_i \geq 0$.

\begin{proof} We view the upper bound as a conjunction of four upper bounds and prove each separately, by applying Lemma \ref{general-contour-integral-bound}.

In each case, by Lemma \ref{general-contour-integral-bound}, it suffices to prove that, for all $\epsilon>0$, the expression \eqref{Euler-bounds-expression}  is $\left( \frac{1}{ 1- q^{\epsilon}}\right)^{O(1)}$  if $\alpha_1,\dots,\alpha_{\rp+\rm}$ satisfy one of the following  \begin{enumerate}

\item For some $i_1 \in S$, we have $ \operatorname{Re} \alpha_{i_1}  \leq \frac{1}{2}-  \epsilon $, and  $\operatorname{Re} \alpha_i \leq - \epsilon$ for all other $i \in S$. For some $i_2 \notin S$, we have $\operatorname{Re} \alpha_{i_2} \geq -\frac{1}{2} + \epsilon $, and  $\operatorname{Re} \alpha_i \geq \epsilon$ for all other $i \notin S$. 

\item  $\operatorname{Re} \alpha_i \leq \frac{1}{2} - \epsilon$ for $i \in S$,  $\operatorname{Re} \alpha_i \geq   \epsilon$ for  $i \notin S$ 

\item  $\operatorname{Re} \alpha_i \leq  - \epsilon$ for $i \in S$,  $\operatorname{Re} \alpha_i \geq  \frac{1}{2} + \epsilon$ for  $i \notin S$

\item $\operatorname{Re} \alpha_i \leq   \frac{1}{2}  - \epsilon$ for all $i \in S$,  $\operatorname{Re} \alpha_i \geq  \frac{1}{2} + \epsilon$ for $i \notin S$.

\item $\operatorname{Re} \alpha_i \leq   -\frac{1}{2}  - \epsilon$ for all $i \in S$,  $\operatorname{Re} \alpha_i \geq - \frac{1}{2} + \epsilon$ for $i \notin S$.

\end{enumerate}

%Indeed, we take $\epsilon = \frac{1}{ (  1+ \sum_{i \in S} d_i - \sum_{i \notin S} d_i )\log q } $ and then integrate over all $\alpha_1,\dots,\alpha_{\rp+\rm}$ exactly attaining the inequality. 

For all cases, we will use the Euler products

 \begin{equation}\label{analysis-sum} \sum_{ \substack{ f_1,\dots, f_{\rp+\rm} \in \mathbb F_q[T] \\ \textrm{monic} \\ \prod_{i \in S} f_i  / \prod_{i\notin S} f_i \in T^{\mathbb Z}} }  \prod_{i \in S} q^{ (-1/2 + \alpha_i)  \deg f_i }  \prod_{i\notin S} q^{( -1/2 - \alpha_i ) \deg f_i  }  \end{equation}
\begin{equation}\label{analysis-L-product} =  \left( \prod_{i \in S} ( 1 - q^{ -1/2 + \alpha_i})^{-1} \prod_{i \notin S}  (1 - q^{ -1/2 - \alpha_i})^{-1}  \right)  \prod_{ \substack { \pi \in \mathbb F_q[T] \\ \textrm{monic} \\ \textrm{irreducible} \\ \pi \neq T }}  \sum_{\substack{ e_1,\dots, e_{\rp+\rm} \in \mathbb N \\  \sum_{i\in S} e_i = \sum_{i \notin S} e_i}} |\pi|^{ - \sum_{i\in S} e_i +   \sum_{i \in S} \alpha_i e_i - \sum_{i \notin S} \alpha_i e_i} \end{equation}

where the first term is the Euler factor at $T$, and \begin{equation}\label{analysis-easy-product} (1 - q^{\alpha_{i_1} - \alpha_{i_2}})  =\prod_{ \substack { \pi \in \mathbb F_q[T] \\ \textrm{monic} \\ \textrm{irreducible} }}  (1- |\pi| ^{ \alpha_{i_1} - \alpha_{i_2} - 1} ).\end{equation} 

Let us briefly discuss issues of convergence. The sum \eqref{analysis-sum} is (absolutely) convergent whenever $\operatorname{Re} \alpha_i <-1/2$ for $i \in S$ and $\operatorname{Re} \alpha_i> 1/2$ for $i \not\in S$. (In fact, this holds even without the condition $ \prod_{i \in S} f_i/ \prod_{i \notin S} f_i \in T^{\mathbb Z}$.) Hence \eqref{analysis-L-product} converges to the value of the sum in the same region. The Euler product \eqref{analysis-easy-product} is valid when $ \operatorname{Re} \alpha_{i_1} - \operatorname{Re} \alpha_{i_2} < 0 $. So the holomorphic function \eqref{Euler-bounds-expression} is equal to 
\begin{equation}\label{analysis-combined-product}  \left( \prod_{i \in S} ( 1 - q^{ -1/2 + \alpha_i})^{-1} \prod_{i \notin S}  (1 - q^{ -1/2 - \alpha_i})^{-1} \prod_{i_1 \in S} \prod_{i_2 \notin S}   ( 1-  q ^{\alpha_{i_1 }- \alpha_{i_2} -1} ) \right) \prod_{ \substack { \pi \in \mathbb F_q[T] \\ \textrm{monic} \\ \textrm{irreducible} \\ \pi \neq T }} f_\pi(\alpha_1,\dots, \alpha_{\rp+\rm} ) \end{equation} as long as $\operatorname{Re} (\alpha_i)$ is sufficiently small for $i\in S$ and sufficiently large for $i \not\in S$, where  \[ f_\pi(\alpha_1,\dots, \alpha_{\rp+\rm} )=  \prod_{i_1 \in S} \prod_{i_2 \notin S}   ( 1-  q ^{\alpha_{i_1} - \alpha_{i_2} -1} )\sum_{\substack{ e_1,\dots, e_{\rp+\rm} \in \mathbb N \\  \sum_{i\in S} e_i = \sum_{i \notin S} e_i}} |\pi|^{ - \sum_{i\in S} e_i +   \sum_{i \in S} \alpha_i e_i - \sum_{i \notin S} \alpha_i e_i} \] We will show that \eqref{analysis-combined-product} in fact converges in a larger region, which will suffice to apply Lemma \ref{general-contour-integral-bound}.

Let us fix $\epsilon$ for the remainder of the proof.

In all five cases, the Euler factors at $T$ is manifestly $O \left(  (1- q^{-\epsilon})^{-O(1)} \right)$ so we focus on the other Euler factors, where it suffices to prove that for $\alpha_1,\dots, \alpha_{\rp+\rm}$ in these ranges

\begin{equation}\label{euler-factor-bound}   \left| f_\pi(\alpha_1,\dots, \alpha_{\rp+\rm} )\right| \leq  ( 1- |\pi|^{ - 1- \epsilon})^{-O(1) } .\end{equation}

This is sufficient because \[ \prod_{ \substack { \pi \in \mathbb F_q[T] \\ \textrm{monic} \\ \textrm{irreducible} \\ \pi \neq T }}  (1 - |\pi|^{-1 -\epsilon} )^{-O(1) } \leq \prod_{ \substack { \pi \in \mathbb F_q[T] \\ \textrm{monic} \\ \textrm{irreducible}  }}  (1 - |\pi|^{-1 -\epsilon} )^{-O(1) } = (\zeta_{\mathbb F_q[T]} ( 1+ \epsilon) )^{O(1) } = ( 1 - q^{ -\epsilon} )^{-O(1) }. \]

We have \begin{equation}\label{euler-factor-sum} f_\pi(\alpha_1,\dots, \alpha_{\rp+\rm}  ) =  \sum_{\substack{ e_1,\dots, e_{\rp+\rm} \in \mathbb N \\  \sum_{i\in S} e_i = \sum_{i \notin S} e_i}} \Bigl( |\pi|^{ - \sum_{i \in S} e_i +   \sum_{i \in S} \alpha_i e_i - \sum_{i \notin S} \alpha_i e_i }  \sum_{\substack { J \subseteq S \times S^c \\ |J \cap p_S^{-1}( j) | \leq e_j  , j \in S \\  | J \cap p_{S^c}^{-1}(j)  | \leq e_j , j \notin S } } (-1)^{ |J|}\Bigr) \end{equation} where $p_S$ and $p_{S^c}$ are the projections onto $S$ and $S^c$.

Observe first that when $e_1,\dots,e_{\rp+\rm}$ are not all zero, \begin{equation}\label{euler-factor-coefficient} \sum_{\substack { J \subseteq S \times S^c \\ |J \cap p_S^{-1}( j) | \leq e_j  , j \in S \\  | J \cap p_{S^c}^{-1}(j)  | \leq e_j , j \notin S } } (-1)^{ |J|} \end{equation} vanishes unless $\max_{i \in S} e_i  < | \{ i \notin S | e_i >0\} |$ or $\max_{i \notin S} e_i <   | \{ i \in S | e_i>0 \} |$. Indeed, if neither of these is satisfied, letting $i_1\in S$ and $i_2\notin S$ be such that $e_{i_1}$ and $e_{i_2}$ attain their maximal values, we see that adding $(i_1,i_2)$ to $J$ or removing it from $J$ preserves the conditions $|J \cap p_S^{-1} (j) | \leq e_j  , j \in S , | J \cap p_{S^c}^{-1}(j) | \leq e_j , j \notin S $, so defines a sign-reversing involution of $J$, and thus the sum vanishes.

In particular, \eqref{euler-factor-coefficient} vanishes for all but finitely many $e_1, \dots, e_{\rp+\rm}$. So in each of the five cases, to prove \eqref{euler-factor-bound} it suffices to show that each term in \eqref{euler-factor-sum}, except the one where $e_1,\dots,e_{\rp+\rm}=0$, is $O( |\pi|^{-1-\epsilon})$.  Furthermore, if $\sum_{i \in S} e_i = \sum_{i \not \in S} e_i =1$, the sum over $J$ has two terms which cancel each other, so we may assume $\sum_{i\in S} e_i \geq 2$.

Case (1) is the most difficult. However, using the inequality we have verified, it is straightforward. Suppose  \[ \max_{i \in S} e_i < | \{ i \notin S | e_i >0\} |,\] so \[\max_{i \in S} e_i  \leq  | \{ i \notin S | e_i >0\} |-1.\]  But\[| \{ i \notin S | e_i >0\} | + \max_{i \notin S} e_i -1 \leq \sum_{i \notin S} e_i,\] so  \[ \max_{i \in S} e_i + \max_{i \notin S} e_i \leq \sum_{i \notin S} e_i\] thus

\[  - \sum_{i \in S}  e_i +   \sum_{i \in S}\operatorname{Re}( \alpha_i )e_i - \sum_{i \notin S} \operatorname{Re}(\alpha_i )e_i  \leq  - \sum_{i\in S} e_i +   \frac{ \max_{i \in S} e_i + \max_{i \notin S} e_i} {2}  - 2\epsilon \leq  - \sum_{i \in S} e_i + \frac{ \sum_{i \notin S} e_i }{2} - 2\epsilon\] \[ \leq - \frac{ \sum_{i \in S} e_i}{2} - 2\epsilon \leq -1 -2 \epsilon \]
so all terms are $O ( |\pi|^{-1 - 2 \epsilon})$.

For cases (2) and (3), we have  \[  |\pi|^{ - \sum_{i \in S } e_i +   \sum_{i \in S}\operatorname{Re}( \alpha_i )e_i - \sum_{i \notin S} \operatorname{Re}(\alpha_i) e_i } \leq |\pi|^{ - \left(\frac{1}{2} + 2 \epsilon\right) \sum_{i \in S} e_i },\] so the terms are $O( |\pi|^{-1-4 \epsilon})$.

For case (4) and (5) , we have  \[  |\pi|^{ - \sum_{i \in S } e_i +   \sum_{i \in S} \operatorname{Re}(\alpha_i )e_i - \sum_{i \notin S} \operatorname{Re}(\alpha_i )e_i } \leq |\pi|^{ - \left(1+ 2 \epsilon\right) \sum_{i \in S} e_i },\] so all the terms are  $O( |\pi| ^{ - 2 -4 \epsilon})$.

\end{proof} 

\begin{lemma}\label{shifted-Euler-bounds}  Let $d_1,\dots, d_{\rp+\rm}$ be integers satisfying the inequalities of Lemma~\ref{M-support}. Then the coefficient of $\prod_i q^{\alpha_i d_i}$  in $ M_S $ is bounded by \[   O \left( \left(O(1) + \sum_{i \in S} d_i - \sum_{i \notin S} d_i \right)^{O(1)} \right) \] 
times 
\[ \min \left( q^{ \frac{ - \max_{i \in S} d_i + \min_{i \notin S} d_i -1}{2} }, q^{ \frac{ - \sum_{i \in S}  d_i  +  { |S| \choose 2} }{2}}, q^{ \frac{  \sum_{ i\notin S}  d_i  +  { |S| \choose 2} - {\rp+\rm \choose 2} }{2}},  q^{- \frac{ \left| \sum_{i \in S}  d_i + \sum_{i \notin S} d_i  - {\rp+\rm \choose 2} \right|}{2} } \right)  . \] 

\end{lemma}

\begin{proof}  %Observe that \[\left( \prod_{i_1 \in S} \prod_{i_2 \notin S} (1 - q^{\alpha_{i_1} - \alpha_{i_2}}) \right) = \left( \prod_{i_1 \in S} \prod_{i_2 \in S} (q^{\alpha_{i_1}} - q^{\alpha_{i_2}})  \right)  \left( \prod_{i \notin S} q^{ - |S| \alpha_i }\right)\] and so 

By Equations \eqref{MS-observation-1} and \eqref{MS-observation-2}, $M_S$ is equal to the expression \eqref{Euler-bounds-expression} times \[  \pm\prod_{\substack{ 1 \leq i_1< i_2 \leq \rp+\rm\\ i_1, i_2 \in S \textrm { or } i_1,i_2 \notin S}} (q^{\alpha_{i_1}} - q^{\alpha_{i_2}})   \prod_{i\notin  S} q^{ |S| \alpha_i}  .\] This additional factor has bounded coefficients and is supported on those terms $q^{ \sum_i \alpha_i d_i}$ where $\{d_i | i \in S\} = \{0,\dots, |S|-1\}$ and $\{ d_i| i \notin S\} = \{|S|,\dots, \rp+\rm-1 \}$. 

Hence we can obtain bounds for $M_S$ by subtracting from the exponents in Lemma~\ref{Euler-bounds} the minimal possible contribution of an element in the support of this additional factor to the exponent, which are as stated. In fact, in all cases but the first, we are minimizing a constant function.

\end{proof}

\begin{lemma}\label{cancellation-agreement}  Assume $|S|  -\rp $ is a multiple of $m$.   The coefficients of $q^{ \sum_i \alpha_i d_i}$ in the power series   \[ \prod_{1 \leq i_1< i_2 \leq \rp+\rm} (q^{\alpha_{i_1}} - q^{\alpha_{i_2}}) \sum_{\chi \in S_{n,q}}  \epsilon_\chi^{-\rm} \prod_{i=1}^{\rp+\rm} L(1/2- \alpha_i, \chi) \]   and \[  (q^{n}- q^{n-1})  \mu^{ \frac{\rp- |S|}{m}}  \prod_{i \notin S} q^{ \alpha_i (n-1)} M_S(\alpha_1,\dots, \alpha_{\rp+\rm} ) \]    agree as long as

\[ \sum_{i \in S } d_i  - { |S| \choose 2} ,  \sum_{i \notin S}  (n-1-d_i) + { \rp+\rm \choose 2} - {|S| \choose 2} \leq n-1 .\]

\end{lemma}

\begin{proof} We have \[\epsilon_\chi^{-\rm}\prod_{i=1}^{\rp+\rm} L(1/2- \alpha_i, \chi)  =\epsilon_\chi^{ \rp+\rm - |S| - \rm} \prod_{i\in S} L(1/2 -\alpha_i, \chi) \prod_{i \notin S} q^{ (n-1) \alpha_i} L (1/2 + \alpha_i, \overline{\chi})  .\] Because $\rp-|S|$ is divisible by $m$, $\epsilon_\chi^{\rp -|S|} = \mu^{ \frac{ \rp-|S|}{m}}$.  Thus 

\[\sum_{\chi\in S_{n,q}} \epsilon_\chi^{-\rm}\prod_{i=1}^{\rp+\rm} L(1/2- \alpha_i, \chi)  \] \[  =\mu^{ \frac{ \rp-|S|}{m}} \prod_{i\notin S} q^{ (n-1) \alpha_i } \sum_{ \substack{ f_1,\dots, f_{\rp+\rm} \in \mathbb F_q[T]\\ \textrm{monic}  }}  \sum_{\chi\in S_{n,q}} \chi\left( \prod_{i \in S} f_i\right) \overline{\chi} \left(\prod_{i \notin S} f_i\right)  \prod_{i \in S} q^{ (-1/2 + \alpha_i)  \deg f_i }  \prod_{i\notin S} q^{( -1/2 - \alpha_i ) \deg f_i  }.\]

Now \[  \sum_{\chi\in S_{n,q}} \chi \left( \prod_{i \in S} f_i\right) \overline{\chi} \left(\prod_{i \notin S} f_i\right)  \] is equal to the sum over all even characters of $\left(\mathbb F_q[x]/x^{n+1}\right)^\times$ minus the sum over all even characters of $\left(\mathbb F_q[x]/x^n \right)^\times$. The characters modulo $x^n$ depend on the leading $n$ coefficients of the polynomial. Hence both sums vanish unless the leading $n$ coefficients are equal.  Thus if $\deg\prod_{i\in S} f_i, \deg \prod_{i \notin S} f_i \leq n-1$, there are only $n$ coefficients, and so both sums cancel unless $\prod_{i \in S} f_i / \prod_{i \notin S} f_i \in T^{\mathbb Z}$, in which case the sum over characters is $q^n - q^{n-1}$.

This occurs precisely in the coefficients $q^ {\sum_i \alpha_i d_i}$ where $\sum_{i \in S} d_i,  \sum_{i \notin S} (n-1-d_i) \leq n-1$.  Hence for $d_i$ satisfying those inequalities, the coefficients of $q^ {\sum_i \alpha_i d_i}$  in $\sum_{\chi\in S_{n,q}} \epsilon_\chi^{-\rm} \prod_{i=1}^{\rp+\rm} L(1/2- \alpha_i, \chi) $ and  \[ (q^n- q^{n-1}) \mu^{ \frac{ \rp-|S|}{m}} \prod_{i\notin S} q^{ (n-1) \alpha_i } \sum_{ \substack{ f_1,\dots, f_{\rp+\rm} \in \mathbb F_q[T]\\ \textrm{monic}\\ \prod_{i\in S} f_i / \prod_{i\notin S} f_i \in T^{\mathbb Z}  }}  \prod_{i \in S} q^{ (-1/2+ \alpha_i) \deg f_i} \prod_{i\notin S} q^{( -1/2 - \alpha_i ) \deg f_i }\] are equal.

Multiplying the monomial $q^ {\sum_i \alpha_i d_i}$  by the Vandermonde determinant produces a sum of monomials. In each monomial,  $\sum_{i \in S} d_i$ is increased by at least $ { |S| \choose 2}$ and $\sum_{i \notin S} (n-1-d_i ) $ is reduced by at most ${ \rp+\rm \choose 2} - {|S| \choose 2}$.  Hence the identity in the multiplied terms is satisfied as long as $\sum_{i \in S } d_i - { |S|\choose 2} \leq n-1$ and $\sum_{i \notin S} (n-1-d_i) + { \rp+\rm \choose 2} - {|S| \choose 2} \leq n-1$.  \end{proof}

\subsection{ Additional results}

Here we prove some additional results that are not necessary to prove Theorem~\ref{main} but may be helpful to interpret it. We describe the general behavior of \eqref{eq-main-right} and compute it in a special case.

 \begin{lemma}\label{leading-term}  \begin{enumerate} 
  
 \item The holomorphic function \eqref{eq-main-right}  evaluated at $\alpha_1 = \dots = \alpha_{\rp+\rm} =0$ (where it is defined by analytic continuation, see Remark~\ref{vandermonde-remark}) is a polynomial in $n$ of degree $\rp\rm$ whose leading term is \begin{equation}\label{CFKRS-polynomial-leading}   a_{\rp,\rm}g_{\rp,\rm} n^{\rp\rm}/(\rp\rm)!  \end{equation} where $g_{\rp,\rm}$ is the random matrix factor \[ g_{\rp,\rm} = (\rp\rm)! \prod_{j=0}^{\rp-1}  j!/ (j+\rm)! \] and $a_{\rp,\rm}$ is the Euler product \begin{equation}\label{CFKRS-arithmetic-factor}  (1- q^{-1})^{\rp\rm} (1- q^{-1/2})^{\rp+\rm}  \prod_{ \substack { \pi \in \mathbb F_q[T]  \\ \textrm{monic} \\ \textrm{irreducible} \\ \pi \neq T}} \left( ( 1- |\pi|^{-1} )^{\rp\rm} \sum_{e \in \mathbb N} { e + \rp-1 \choose \rp-1} {e + \rm-1 \choose \rm-1}  |\pi|^{-e} \right).\end{equation}
(In this Euler product, $(1- q^{-1})^{\rp\rm} (1- q^{-1/2})^{\rp+\rm} $ may be viewed as the Euler factor at $T$. )

 \item  On the portion of the imaginary axis where $\alpha_1,\dots,\alpha_{\rp+\rm}$ are distinct modulo $2\pi i /\log q$, the individual terms \[ \sum_{ \substack{ f_1,\dots, f_{\rp+\rm} \in \mathbb F_q[T] \\ \textrm{monic} \\ \prod_{i \in S} f_i / \prod_{i\notin S} f_i \in T^{\mathbb Z} }}  \prod_{i\in S} |f_i|^{ -\frac{1}{2} +\alpha_i} \prod_{i \notin S} |f_i|^{ - \frac{1}{2} - \alpha_i}\] of \eqref{eq-main-right} are holomorphic so the sum \eqref{eq-main-right} is simply a linear combination of terms $\prod_{i \notin S} q^{ \alpha_i (n-1)} $, i.e. a quasiperiodic function of $n$. \end{enumerate}
 \end{lemma}

\begin{proof} It follows from Lemma~\ref{Euler-bounds}  that \begin{equation}\label{holomorphic-mt-factor} \left( \prod_{i_1 \in S} \prod_{i_2 \notin S}\left(1 - q^{\alpha_{i_1} - \alpha_{i_2}}\right)\right) \sum_{ \substack{ f_1,\dots, f_{\rp+\rm} \in \mathbb F_q[T] \\ \textrm{monic} \\ \prod_{i \in S} f_i  /  \prod_{i\notin S} f_i \in T^{\mathbb Z}} }  \prod_{i \in S} q^{ (-1/2 + \alpha_i)  \deg f_i }  \prod_{i\notin S} q^{( -1/2 - \alpha_i ) \deg f_i  } ,\end{equation} when written as a power series in $q^{-\alpha_i}$, converges for $\alpha_i$ in a neighborhood of the imaginary axis. Thus  \[ \sum_{ \substack{ f_1,\dots, f_{\rp+\rm} \in \mathbb F_q[T] \\ \textrm{monic} \\ \prod_{i \in S} f_i  /  \prod_{i\notin S} f_i \in T^{\mathbb Z}} }  \prod_{i \in S} q^{ (-1/2 + \alpha_i)  \deg f_i }  \prod_{i\notin S} q^{( -1/2 - \alpha_i ) \deg f_i  } \] is equal to the quotient of a holomorphic function in this neighborhood by \[ \left( \prod_{i_1 \in S} \prod_{i_2 \notin S}\left(1 - q^{\alpha_{i_1} - \alpha_{i_2}}\right)\right) .\] 

We will use that property to prove both parts of this lemma.

To prove part (2), if $\alpha_1,\dots,\alpha_{\rp+\rm}$ are distinct modulo $2\pi i/\log q$ then the denominators $ \prod_{i_1 \in S} \prod_{i_2 \notin S}\left(1 - q^{\alpha_{i_1} - \alpha_{i_2}}\right)$ are nonvanishing and so the terms in the sum over $S$ in Theorem \ref{main} are individually holomorphic in a neighborhood of $\alpha_1,\dots,\alpha_{\rp+\rm}$. This means summing them as meromorphic functions and analytically continuing is the same as summing them normally. Thus \eqref{eq-main-right} is simply a linear combination of the functions $\prod_{i \notin S} q^{ \alpha_i (n-1)} $ for distinct subsets $S$, i.e. it is a quasiperiodic function of $n$.

To prove part (1), it is simplest to view each $\alpha_i$ as a linear function of a distinct variable $t$ and set $t=0$. Because each term $(1 - q^{\alpha_{i_1} - \alpha_{i_2}})$ vanishes to order $1$ at $t=0$, the product $ \prod_{i_1 \in S} \prod_{i_2 \notin S}\left(1 - q^{\alpha_{i_1} - \alpha_{i_2}}\right)$ vanishes to order $|S| (\rp+\rm-|S|) = \rp\rm$  at $t=0$. Thus  \[ \sum_{ \substack{ f_1,\dots, f_{\rp+\rm} \in \mathbb F_q[T] \\ \textrm{monic} \\ \prod_{i \in S} f_i  /  \prod_{i\notin S} f_i \in T^{\mathbb Z}} }  \prod_{i \in S} q^{ (-1/2 + \alpha_i)  \deg f_i }  \prod_{i\notin S} q^{( -1/2 - \alpha_i ) \deg f_i  } \] is a holomorphic function of $t$ divided by $t^{\rp\rm}$. On the other hand, the coefficient of $t^j$ in the Taylor series for $\prod_{i \notin S} q^{ \alpha_i (n-1)}$ is a polynomial in $n$ of degree $\leq j$. Multiplying by a holomorphic function in $t$ preserves this property, as does summing, so all told the main term is a power series in $t$ whose $j$th coefficient is a polynomial in $n$ of degree $\leq j$, all divided by $t^{\rp\rm}$. The value at $t=0$ is the coefficient of $t^{\rp\rm}$ in the numerator, which is a polynomial in $n$ of degree $\leq \rp\rm$. 

Furthermore, we can calculate the leading coefficient of this polynomial in $n$. From the proof of polynomiality, we can see that the holomorphic function \eqref{holomorphic-mt-factor} enters into the calculation only by its value when $\alpha_1,\dots,\alpha_{\rp+\rm}=0$. This value is independent of $S$, by permuting the variables.  Call it $a_{\rp,\rm}$. Because only the value at $\alpha_1,\dots,\alpha_{\rp+\rm}=0$ is significant, we can replace the term  \[ \sum_{ \substack{ f_1,\dots, f_{\rp+\rm} \in \mathbb F_q[T] \\ \textrm{monic} \\ \prod_{i \in S} f_i  /  \prod_{i\notin S} f_i \in T^{\mathbb Z}} } \prod_{i \in S} q^{ (-1/2 + \alpha_i)  \deg f_i }  \prod_{i\notin S} q^{( -1/2 - \alpha_i ) \deg f_i  } \ \] with \[ \frac{ a_{\rp,\rm}}{  \prod_{i_1 \in S} \prod_{i_2 \notin S}\left(1 - q^{\alpha_{i_1} - \alpha_{i_2}}\right)}  \] without affecting the leading coefficient of the polynomial in $n$. After replacing these terms in \eqref{eq-main-right}, we obtain.

\begin{equation}\label{random-matrix-sum}  a_{\rp,\rm} \sum_{\substack{  S \subseteq \{1,\dots,\rp+\rm\} \\ |S|=r}}  \frac{ \prod_{i \notin S} q^{ \alpha_i (n-1)}}{   \prod_{i_1 \in S} \prod_{i_2 \notin S}\left(1 - q^{\alpha_{i_1} - \alpha_{i_2}}\right)}  \end{equation}

The sum in \eqref{random-matrix-sum} is also the Weyl character formula for the representation of $GL_{\rp+\rm}$ with highest weights $0$ repeated $r$ times and $n-1$ repeated $s$ times, evaluated at the diagonal element with eigenvalues $q^{\alpha_i}$. So when $\alpha_i=0$, the sum in \eqref{random-matrix-sum} is simply the dimension of this representation, which by the Weyl dimension formula is a polynomial in $n$ with leading term $ n ^{\rp\rm} \prod_{j=0}^{\rp-1}  j!/ (j+\rm)! $. Multiplying by $a_{\rp,\rm}$, the leading term in \eqref{random-matrix-sum} is exactly \eqref{CFKRS-polynomial-leading}, except that we have to prove the Euler product formula \eqref{CFKRS-arithmetic-factor} for $a_{\rp,\rm}$. 

To do this, we follow the method of Lemma~\ref{Euler-bounds} and use the Euler products \eqref{analysis-L-product} and \eqref{analysis-easy-product}.

%\[ \sum_{ \substack{ f_1,\dots, f_{\rp+\rm} \in \mathbb F_q[T] \\ \textrm{monic} \\ \prod_{i \in S} f_i  /  \prod_{i\notin S} f_i \in T^{\mathbb Z}} }  \prod_{i \in S} q^{ (-1/2 + \alpha_i)  \deg f_i }  \prod_{i\notin S} q^{( -1/2 - \alpha_i ) \deg f_i  } \] 
%\[=  \left( \prod_{i \in S} ( 1 - q^{ -1/2 + \alpha_i})^{-1} \prod_{i \notin S}  (1 - q^{ -1/2 - \alpha_i})^{-1}  \right)  \prod_{ \substack { \pi \in \mathbb F_q[T] \\ \textrm{monic} \\ \textrm{irreducible} \\ \pi \neq T }}  \sum_{\substack{ d_1,\dots, d_{\rp+\rm} \in \mathbb N \\  \sum_{i\in S} d_i = \sum_{i \notin S} d_i}} |\pi|^{ - \sum_{i\in S} d_i +   \sum_{i \in S} \alpha_i d_i - \sum_{i \notin S} \alpha_i d_i}  .\]

When we combine these two Euler products, the factor for every $\pi \neq T$ is \[ \left(  \prod_{i_1 \in S, i_2\notin S} (1 - |\pi|^{-1+ \alpha_{i_1} - \alpha_{i_2}} ) \right)  \sum_{\substack{ d_1,\dots, d_{\rp+\rm} \in \mathbb N \\  \sum_{i\in S} d_i = \sum_{i \notin S} d_i}} |\pi|^{ - \sum_{i\in S} d_i +   \sum_{i \in S} \alpha_i d_i - \sum_{i \notin S} \alpha_i d_i}  = 1 + O ( |\pi|^{-2} ) \] so this Euler product converges. When we specialize $\alpha_1,\dots, \alpha_{\rp+\rm}=0$, the Euler factor for $\pi \neq T$ becomes
 \[ (1 - |\pi|^{-1} )^{\rp\rm}  \sum_{\substack{ d_1,\dots, d_{\rp+\rm} \in \mathbb N \\  \sum_{i\in S} d_i = \sum_{i \notin S} d_i}} |\pi|^{ - \sum_{i\in S} d_i } .\]

For each natural number $e$, there are ${e + \rp-1 \choose \rp-1}$ nonnegative integer solutions to $\sum_{i \in S} d_i = e$ and  ${e + \rm-1 \choose \rm-1}$ nonnegative integer solutions to $\sum_{i \notin S} d_i = e$ and  so the coefficient of $\pi^e$ in $ \sum_{\substack{ d_1,\dots, d_{\rp+\rm} \in \mathbb N \\  \sum_{i\in S} d_i = \sum_{i \notin S} d_i}} |\pi|^{ - \sum_{i\in S} d_i } $ is ${ e + \rp-1 \choose \rp-1} {e + \rm-1 \choose \rm-1}$.

When we specialize $\alpha_1,\dots,\alpha_{\rp+\rm}=0$, the Euler factor for $\pi=T$ becomes \[(1- q^{-1})^{\rp\rm} (1- q^{-1/2})^{\rp+\rm}.\] Combining these Euler factors, we get \eqref{CFKRS-arithmetic-factor}.

 \end{proof}
 
 \begin{lemma}\label{fourth-moment}    When $\rp=\rm=2$, \eqref{eq-main-right} specializes to
 
  \begin{equation}\label{fourth-main-term}\sum_{\substack{  S \subseteq \{1,\dots,4\} \\ |S|=2}} \prod_{i \notin S} q^{ \alpha_i (n-1)} \frac{ 1- q^{ - 1 - \sum_{i \in S} \alpha _i + \sum_{i \not \in S } \alpha_i}}{ 1- q^{  -2-  \sum_{i \in S} \alpha _i + \sum_{i \not \in S } \alpha_i}} \prod_{i \in S} \frac{1}{ 1- q^{-1/2 - \alpha_i}} \prod_{i \not \in S} \frac{1}{ 1- q^{-1/2 + \alpha_i}} \prod_{i_1 \in S, i_2 \not \in S} \frac{ 1- q^{ - 1 -\alpha_{i_1 } + \alpha_{i_2}}}{1 - q^{-\alpha_{i_1} + \alpha_{i_2}}}.\end{equation}

 \end{lemma}
 
 \begin{proof} Let us consider first the term where $S = \{1,2\}$.
 
 We have 
 
 \begin{equation}\label{fourth-polynomial-sum} \sum_{ \substack{ f_1,f_1,f_3,f_4 \in \mathbb F_q[T] \\ \textrm{monic} \\ f_1f_2/f_3f_4 \in T^{\mathbb Z} }}  |f_1|^{-\frac{1}{2} - \alpha_1} |f_1|^{ - \frac{1}{2} - \alpha_2} |f_3|^{-\frac{1}{2} + \alpha_3 } |f_4| ^{ -\frac{1}{2} + \alpha_4} \end{equation}
 \[ = \frac{\sum_{ \substack{ f_1,f_1,f_3,f_4 \in \mathbb F_q[T] \\ \textrm{monic} \\ \textrm{prime to } T\\ f_1f_2=f_3f_4}}  |f_1|^{-\frac{1}{2} - \alpha_1} |f_1|^{ - \frac{1}{2} - \alpha_2} |f_3|^{-\frac{1}{2} + \alpha_3 } |f_4| ^{ -\frac{1}{2} + \alpha_4}} { (1 - q^{-\frac{1}{2} - \alpha_1}) (1 - q^{-\frac{1}{2} - \alpha_2}) (1 - q^{-\frac{1}{2} +\alpha_3})(1 - q^{-\frac{1}{2} + \alpha_4 })}  \]
 
 Now

 \begin{equation}\label{fourth-noT-sum}   \hspace*{-1cm} \sum_{ \substack{ f_1,f_1,f_3,f_4 \in \mathbb F_q[T] \\ \textrm{monic} \\ \textrm{prime to } T\\ f_1f_2=f_3f_4}  } |f_1|^{-\frac{1}{2} - \alpha_1}|f_2|^{ - \frac{1}{2} - \alpha_2} |f_3|^{-\frac{1}{2} + \alpha_3 } |f_4| ^{ -\frac{1}{2} + \alpha_4} = \prod_{ \substack { \pi \in \mathbb F_q[T] \\ \textrm{monic} \\ \textrm{irreducible} \\ \pi \neq T} }  \sum_{\substack{ e_1,e_2,e_3,e_4 \in \mathbb N \\ e_1 + e_2 = e_3 + e_4}} |\pi| ^{ - \frac{e_1+e_2+e_3+e_4}{2} - \alpha_1 e_1 - \alpha_2 e_2 + \alpha_3 e_3 + \alpha_4 e_4}  \end{equation}
 
 Given $e_1,e_2,e_3,e_4 \in \mathbb N$ with $e_1+ e_2 = e_3 +e_4$, the number of ways of writing $e_1 = a + b$, $e_2 =c+d$, $e_3 = a+c$, $e_4 = b+d$ with $a,b,c,d \in \mathbb N$ is $\min(e_1,e_2,e_3,e_4) + 1$ as we must have $b= e_1 -a $, $c = e_3-a$, $d = e_2-e_3 +a = e_4 - e_1 + a$ and the valid $a$ are in the interval $[ \max(e_3-e_2,0), \min(e_1,e_3)]$ whose length is $\min(e_1,e_2,e_3,e_4)$.
 
 Hence the number of solutions for $e_1,e_2,e_3,e_4$ minus the number of solutions for $e_1-1,e_2-1,e_3-1,e_4-1$ is exactly $1$ if $e_1,e_2,e_3,e_4$ are nonnegative with $e_1 + e_2 = e_3 + e_4$ and zero otherwise. This gives 
 
 \[ \sum_{\substack{ e_1,e_2,e_3,e_4 \in \mathbb N \\ e_1 + e_2 = e_3 + e_4}} |\pi| ^{ - \frac{e_1+e_2+e_3+e_4}{2} - \alpha_1 e_1 - \alpha_2 e_2 + \alpha_3 e_3 + \alpha_4 e_4} \] \[ = \frac{ 1 - |\pi|^{ -2 - \alpha_1 - \alpha_2 + \alpha_3 + \alpha_4 }}{ (1 - |\pi|^{-1- \alpha_1 +\alpha_3}) (1- |\pi|^{-1- \alpha_1 + \alpha_4}) (1- |\pi|^{-1-\alpha_2 + \alpha_3}) ( 1- |\pi|^{ -1-\alpha_2 +\alpha_4})}.\]
 
 Hence using the zeta function of $\mathbb F_q[T]$, \eqref{fourth-noT-sum} is 
 \[\frac{ 1 - q^{ -1 - \alpha_1 - \alpha_2 + \alpha_3 + \alpha_4}}  { 1 - q^{ -2 - \alpha_1 - \alpha_2 + \alpha_3 + \alpha_4}}\prod_{i_1 \in \{1,2\}, i_2 \in \{3,4\} } \frac{ 1- q^{ - 1 -\alpha_{i_1 } + \alpha_{i_2}}}{1 - q^{-\alpha_{i_1} + \alpha_{i_2}}}\]

 and thus \eqref{fourth-polynomial-sum} is  
 \[ \frac{1} { (1 - q^{-\frac{1}{2} - \alpha_1}) (1 - q^{-\frac{1}{2} - \alpha_2}) (1 - q^{-\frac{1}{2} +\alpha_3})(1 - q^{-\frac{1}{2} + \alpha_4 })}   \frac{ 1 - q^{ -1 - \alpha_1 - \alpha_2 + \alpha_3 + \alpha_4}}  { 1 - q^{ -2 - \alpha_1 - \alpha_2 + \alpha_3 + \alpha_4}}\prod_{i_1 \in \{1,2\}, i_2 \in \{3,4\} } \frac{ 1- q^{ - 1 -\alpha_{i_1 } + \alpha_{i_2}}}{1 - q^{-\alpha_{i_1} + \alpha_{i_2}}}.\]
 
 Now if $S \neq \{1,2\}$, we get the same formula, except with the variables permuted by some fixed permutation $\sigma \in S_4$ sending $\{1,2\}$ to $S$. Summing over the possible choices of $S$, we obtain \eqref{fourth-main-term}.

  \end{proof}

\section{Conclusion}

%\begin{lemma}\label{only-satisfied-once}  Assume that the characteristic of $\mathbb F_q$ is $>2$ or that $n>5$.
%
%
%Assume that $d_1,\dots,d_{\rp+\rm}$ are distinct and lie in the interval $. Then the inequality $\sum_{i \in S} d_i + \sum_{i \notin S} (n-1-d_i) - 2 {|S|\choose 2} + {\rp+\rm \choose 2}  \leq n-1$ is satisifed for at most one $S$ with $r-|S|$ divisible by $m$. \end{lemma}
%
%\begin{proof} We have \[ \left( \sum_{i \in S} d_i + \sum_{i \notin S} (n-1-d_i) \right) + \left(  \sum_{i \in S'} d_i + \sum_{i \notin S'} (n-1-d_i)  \right) \] \[ \geq 2 { |S \cap S'|\choose 2}  -2 {\rp+\rm\choose 2} + 2{S \cup S' \choose 2} +    (n-1) | S \oplus S' | \] so \[ \left( \sum_{i \in S} d_i + \sum_{i \notin S} (n-1-d_i) - 2 {|S|\choose 2} + {\rp+\rm \choose 2}  \right) + \left( \sum_{i \in S'} d_i + \sum_{i \notin S'} (n-1-d_i) - 2 {|S'|\choose 2} + {\rp+\rm \choose 2} \right)\] \[  \geq (n-1) |S \oplus S'| - 2 { |S| \choose 2} -2 {|S'|\choose 2} + 2{| S \cap S'| \choose 2} +2 { |S \cup S'| \choose 2} \] \[ = (n-1) |S \oplus S'|  + 2 | S \setminus S' | | S' \setminus S|   . \]  Because $m> 2$ by Lemma \ref{m-estimate}, we have $|S \oplus S'| \geq 2$ and $ |S \oplus S' | \geq 3$ if $|S \setminus S'| |S' \setminus S|=0$, so this is strictly greater than $2n-2$, hence only one of the summands can at most $n-1$ (as they are both nonnegative). \end{proof}
%

We prove here a slightly more general version of the main theorem.

\begin{prop}\label{main1} Assume $n\geq 3$, if $n=3$ that the characteristic of $\mathbb F_q$ is not $2$ or $5$, and if $n=4$ or $5$ that the characteristic of  $\mathbb F_q$ is not $2$.

Assume Hypothesis $\operatorname{H}(n,\rp,\rm,w)$. Let $\alpha_{1},\dots,\alpha_{\rp+\rm}$ be imaginary. Let $C_{\rp,\rm}=\left(\max(\rp,\rm)+2\right)^{\max(\rp,\rm)+1}$  .  Then

\begin{equation}\label{main1-left} \prod_{1 \leq i_1< i_2 \leq \rp+\rm} (q^{\alpha_{i_1}} - q^{\alpha_{i_2}}) \sum_{\chi\in S_{n,q}} \epsilon_\chi^{-\rm}  \prod_{i=1}^{\rp+\rm} L(1/2- \alpha_i, \chi)  \end{equation}

\begin{equation}\label{main1-right} =  \sum_{\substack{  S \subseteq \{1,\dots,\rp+\rm\} \\ m | \rp-|S|}}(q^{n}- q^{n-1}) \mu^{ \frac{r- |S|}{m}}  \prod_{i \notin S} q^{ \alpha_i (n-1)} M_S(\alpha_1,\dots, \alpha_{\rp+\rm})  \end{equation}

\begin{equation} \label{main1-error}  + O \left(  \left( \prod_{1 \leq i_1< i_2 \leq \rp+\rm} \left| q^{\alpha_{i_1}} - q^{\alpha_{i_2}} \right|\right)    n^{\rp+\rm}  C_{\rp+\rm}^{n-1} q^{\frac{n+w}{2}}\right). \end{equation}

 \end{prop}

\begin{proof} For $d_1,\dots, d_{\rp+\rm} \in \mathbb Z$, let $T(d_1,\dots, d_{\rp+\rm})$ be the coefficient of $\prod_{i=1}^{\rp+\rm} q^{ \alpha_i  d_i}$ in \eqref{main1-left} and let $R( d_1,\dots, d_{\rp+\rm})$ be the coefficient of $\prod_{i=1}^{\rp+\rm}  q^{ \alpha_i  d_i}$ in \eqref{main1-right}.

The two sides \eqref{main1-left} and \eqref{main1-right} are antisymmetric in the variables $\alpha_1,\dots,\alpha_{\rp+\rm}$. For \eqref{main1-left} this is clear and for \eqref{main1-right}, this follows from Lemma~\ref{M-support}.  Thus we can write \eqref{main1-left} as
\[ \sum_{d_1< \dots <  d_{\rp+\rm} }  \sum_{\sigma \in S_{\rp+\rm} }  \sign(\sigma) T (d_1,\dots, d_{\rp+\rm}) q^{ \sum_{i=1}^{\rp+\rm} \alpha_i d_{\sigma(i)}}\]
and \eqref{main1-right} as 
\[ \sum_{d_1< \dots <  d_{\rp+\rm} }  \sum_{\sigma \in S_{\rp+\rm} }  \sign(\sigma) R (d_1,\dots, d_{\rp+\rm}) q^{ \sum_{i=1}^{\rp+\rm} \alpha_i d_{\sigma(i)}}\]
so the difference of \eqref{main1-left} and \eqref{main1-right} is
\[ \sum_{d_1< \dots <  d_{\rp+\rm} }  \left( T (d_1,\dots, d_{\rp+\rm}) - R (d_1,\dots, d_{\rp+\rm})\right)  \sum_{\sigma \in S_{\rp+\rm} }  \sign(\sigma)q^{ \sum_{i=1}^{\rp+\rm} \alpha_i d_{\sigma(i)}}. \]

By the Weyl character formula, \[ \frac{  \sum_{\sigma \in S_{\rp+\rm} }\sign(\sigma) q^{ \sum_{i=1}^{\rp+\rm} \alpha_i d_{\sigma(i)}}} {q^{\alpha_{i_1}} - q^{\alpha_{i_2}}  } \] is trace of a diagonal matrix with entries $q^{\alpha_1},\dots, q^{\alpha_{\rp+\rm}}$ on the irreducible representation of $GL_{\rp+\rm}$ with highest weight vector $(d_{\rp+\rm}+1-\rp-\rm, \dots, d_2-1,d_1 )$ \cite[Theorem 24.2]{FH}. Because this diagonal matrix has eigenvalues of norm $1$, the absolute value of its trace is at most the dimension of that highest weight representation, which is $ \prod_{1 \leq i_1< i_2 \leq \rp+\rm} (d_{i_2}-d_{i_1} )$ \cite[Corollary 24.6]{FH}. This implies the estimate 

\begin{equation} \label{weyl-estimate} \left| \sum_{\sigma \in S_{\rp+\rm} }\sign(\sigma) q^{ \sum_{i=1}^{\rp+\rm} \alpha_i d_{\sigma(i)}} \right|   \leq  \prod_{1\leq i_1< i_2 \leq \rp+\rm} | q^{\alpha_{i_1}} - q^{\alpha_{i_2}} |\frac{  \prod_{1 \leq i_1< i_2 \leq \rp+\rm} |d_{i_1}-d_{i_2} | }{ \prod_{i=1}^{\rp+\rm} (i-1)!} \end{equation}
 
Using \eqref{weyl-estimate}, it suffices to prove that \[ \sum_{d_1< \dots <  d_{\rp+\rm} } \left|  T (d_1,\dots, d_{\rp+\rm}) - R (d_1,\dots, d_{\rp+\rm}) \right| \frac{  \prod_{1 \leq i_1< i_2 \leq \rp+\rm} |d_{i_1}-d_{i_2} | }{ \prod_{i=1}^{\rp+\rm} (i-1)!}   \] \begin{equation}\label{main1-error-modified} = O \left(  n^{\rp+\rm}  C_{\rp+\rm}^{n-1} q^{\frac{n+w}{2}}\right) \end{equation}

Let $R_S (d_1,\dots, d_{\rp+\rm}) $ be the coefficient of $\prod_{i=1}^{\rp+\rm} q^{ \alpha_i  d_i}$ in \[(q^{n}- q^{n-1}) \ \mu^{ \frac{r- |S|}{m}}  \prod_{i \notin S} q^{ \alpha_i (n-1)} M_S(\alpha_1,\dots, \alpha_{\rp+\rm}) \] so that \begin{equation}\label{R-decomposition} R (d_1,\dots, d_{\rp+\rm}) = \sum_{\substack{  S \subseteq \{1,\dots,\rp+\rm\} \\ m | \rp-|S|}} R_S(d_1,\dots, d_{\rp+\rm}) .\end{equation}

By Lemma \ref{cancellation-agreement},  \begin{equation}\label{cancellation-equation} T( d_1,\dots,d_{\rp+\rm} ) = R_S( d_1,\dots, d_{\rp+\rm} )\end{equation} as long as \begin{equation}\label{S-cancellation-range} 0 \leq \sum_{i \in S } d_i  - { |S| \choose 2} , \hspace{15pt} \sum_{i \notin S}  (n-1-d_i) + { \rp+\rm \choose 2} - {|S| \choose 2} \leq n-1. \end{equation} 

We can distinguish three types of tuples $d_1,\dots, d_{\rp+\rm}$. The first is where \eqref{S-cancellation-range} is not satisfied for any $S$, the second where  \eqref{S-cancellation-range} is satisfied for a unique $S$, and the third where  \eqref{S-cancellation-range} is satisfied for more than one $S$.  In the first case, by \eqref{R-decomposition}, we have
\[ \left| T ( d_1,\dots,d_{\rp+\rm} ) - R ( d_1,\dots,d_{\rp+\rm} ) )  \right|  \leq \left| T ( d_1,\dots,d_{\rp+\rm} ) \right|  + \sum_{\substack{  S \subseteq \{1,\dots,\rp+\rm\} \\ m | \rp-|S|}} \left|   R_S ( d_1,\dots,d_{\rp+\rm} ) \right| .\] 
In the second case, by \eqref{R-decomposition} and \eqref{cancellation-equation}, we have
\[ \left| T ( d_1,\dots,d_{\rp+\rm} ) - R ( d_1,\dots,d_{\rp+\rm} ) )  \right|  \leq  \sum_{\substack{  S \subseteq \{1,\dots,\rp+\rm\} \\ m | \rp-|S| \\ \textrm{\eqref{S-cancellation-range} does not hold}}} \left|   R_S ( d_1,\dots,d_{\rp+\rm} ) \right| .\] 
In the third case, by \eqref{R-decomposition} and \eqref{cancellation-equation}, we have
 \[ \left| T ( d_1,\dots,d_{\rp+\rm} ) - R ( d_1,\dots,d_{\rp+\rm} ) )  \right|  \leq  \sum_{\substack{  S \subseteq \{1,\dots,\rp+\rm\} \\ m | \rp-|S| }} \left|   R_S ( d_1,\dots,d_{\rp+\rm} ) \right| \] where we have added the canceled $R_S$ term on the right side back in to simplify the expression, without affecting the validity of the inequality.
 
Combining all these, we have

 \[ \sum_{d_1< \dots <  d_{\rp+\rm} }  \left| T ( d_1,\dots,d_{\rp+\rm} ) - R ( d_1,\dots,d_{\rp+\rm} ) )  \right|  \frac{  \prod_{1 \leq i_1< i_2 \leq \rp+\rm} |d_{i_1}-d_{i_2} | }{ \prod_{i=1}^{\rp+\rm} (i-1)!}  \] 
 \begin{equation}\label{first-type-of-terms}  \leq \sum_{\substack{ d_1< \dots <  d_{\rp+\rm} \\ \textrm{\eqref{S-cancellation-range} does not hold for any }S } } \left| T ( d_1,\dots,d_{\rp+\rm} ) \right| \frac{  \prod_{1 \leq i_1< i_2 \leq \rp+\rm} |d_{i_1}-d_{i_2} | }{ \prod_{i=1}^{\rp+\rm} (i-1)!}   \end{equation}
\begin{equation}\label{second-type-of-terms} +  \sum_{\substack{  S \subseteq \{1,\dots,\rp+\rm\} \\ m | \rp-|S| }}(q^{n}- q^{n-1})  \sum_{\substack{ d_1< \dots <  d_{\rp+\rm} \\ \textrm{\eqref{S-cancellation-range} does not hold} } }   \left|   R_S ( d_1,\dots,d_{\rp+\rm} )\right|  \frac{  \prod_{1 \leq i_1< i_2 \leq \rp+\rm} |d_{i_1}-d_{i_2} | }{ \prod_{i=1}^{\rp+\rm} (i-1)!} \end{equation}
\begin{equation}\label{third-type-of-terms}  +  \sum_{\substack{  S \subseteq \{1,\dots,\rp+\rm\} \\ m | \rp-|S| }}(q^{n}- q^{n-1})  \sum_{\substack{ d_1< \dots <  d_{\rp+\rm} \\ \textrm{\eqref{S-cancellation-range} holds for }S\\ \textrm{\eqref{S-cancellation-range} holds for some }S' \neq S  }}   \left|   R_S ( d_1,\dots,d_{\rp+\rm} ) \right|  \frac{  \prod_{1 \leq i_1< i_2 \leq \rp+\rm} |d_{i_1}-d_{i_2} | }{ \prod_{i=1}^{\rp+\rm} (i-1)!} \end{equation} 

We next will prove in Lemmas \ref{first-type-bound}, \ref{second-type-bound}, and \ref{third-type-bound}, bounds for \eqref{first-type-of-terms}, \eqref{second-type-of-terms}, and \eqref{third-type-of-terms} respectively.  This gives
\[ \eqref{first-type-of-terms}+ \eqref{second-type-of-terms}+ \eqref{third-type-of-terms}\leq  O( n^{O(1)} q^{n/2} ) + O (n^{O(1)} q^{n/2}) + O \left(  n^{\rp+\rm}  C_{\rp+\rm}^{n-1} q^{\frac{n+w}{2}}\right) = O \left(  n^{\rp+\rm}  C_{\rp+\rm}^{n-1} q^{\frac{n+w}{2}}\right)\] 
since $n^{O(1)} = O(C_{\rp,\rm}^{n-1})$, which is the desired bound \eqref{main1-error-modified}.

\end{proof}

\begin{lemma}\label{third-type-range} For every $S \subseteq \{1,\dots, \rp+\rm\}$ with $m| \rp-|S|$, for each tuple $d_1< \dots < d_{\rp+\rm} \in \mathbb Z$ that satisfies \eqref{S-cancellation-range}, if $R_S(d_1,\dots, d_{\rp+\rm})\neq 0$, then $0 \leq d_1 < \dots < d_{\rp+\rm} \leq n+ \rp + \rm-2$. \end{lemma}

\begin{proof} By  \eqref{cancellation-equation}, the assumptions imply that $T(d_1,\dots,d_{\rp+\rm})$ is not zero. By definition $T(d_1,\dots,d_{\rp+\rm})$ is the coefficient of  $q^{\sum_i d_i \alpha_i}$ in the product of $ \prod_{1 \leq i_1< i_2 \leq \rp+\rm} (q^{\alpha_{i_1}} - q^{\alpha_{i_2}})$, which is a polynomial in the $q^{\alpha_i}$ of degree at most $\rp+\rm-1$ in each variable, with $ \sum_{\chi\in S_{n,q}} \epsilon_\chi^{-s}  \prod_{i=1}^{\rp+\rm} L(1/2- \alpha_i, \chi)$, which is a polynomial in the $q^{\alpha_i} $ of degree at most $n-1$ in each variable. Hence this product is a polynomial of degree $\leq n + \rp + \rm-2$ in each variable, and thus the coefficient  $T(d_1,\dots,d_{\rp+\rm})$  can only be nonzero if $0 \leq d_1,\dots, d_{\rp+\rm} \leq n+\rp+\rm-2$.  \end{proof}

\begin{lemma}\label{third-type-individual} For every pair $S,S' \subseteq \{1,\dots, \rp+\rm\}$ with $m| \rp-|S|,$  $m|\rp-|S'|,$ and $S \neq S',$ for each tuple $d_1< \dots < d_{\rp+\rm} \in \mathbb Z$ that satisfies \eqref{S-cancellation-range}, we have $R_S(d_1,\dots, d_{\rp+\rm}) =O \left( n^{O(1)} q^{ n/2} \right) $ \end{lemma}

\begin{proof} We may assume without loss of generality that $R_S(d_1,\dots, d_{\rp+\rm}) \neq 0$. We will apply Lemma~\ref{shifted-Euler-bounds} to bound this term, noting by Lemma. \ref{third-type-range} that the factor $( O(1) + \sum_{i\in S} d_i - \sum_{i\notin S } d_i)^{O(1)}$ appearing in Lemma~\ref{shifted-Euler-bounds} is $n^{O(1)}$. 

We can freely use the consequence of \eqref{cancellation-equation} that \[ R_S ( d_1,\dots,d_{\rp+\rm})= T( d_1,\dots,d_{\rp+\rm})= R_{S'} ( d_1,\dots,d_{\rp+\rm}) \] to pass between $S$ and $S'$.

We will split into two cases depending on if $|S'|=|S|$ or not.

If $|S'|=|S|$, we choose $i_1$ in $S$ but not $S'$ and $i_2$ in $S'$ but not $S$. We have \[ \left( d_{i_1}  + (n-1- d_{i_2} )  \right) + \left( (n-1-d_{i_1}  ) + d_{i_2} \right)  \geq 2(n-1),\] so one is at least $n-1$. Applying the first part of Lemma~\ref{shifted-Euler-bounds} for $S$ if the first one is smaller and $S'$ if the second one is smaller, we see that \[R_S( d_1,\dots,d_{\rp+\rm})   = O \left( n^{O(1)} q^{ n - n/2} \right) .\] 

If $|S'| \neq |S|$, assume without loss of generality that $|S| < |S'|$. Then by Lemma~\ref{m-estimate}, $|S'|- |S| \geq m \geq 3$. Hence  \[ \left( \sum_{i \in S} d_i  - { \rp+\rm \choose 2}  - (\rp+\rm- |S|) (n-1)  \right) - \left(  \sum_{i \in S'} d_i  - { \rp+\rm \choose 2}  - (\rp+\rm- |S'|) (n-1) \right) \] \[\geq 3 (n-1) > 2n-2\] so one of these two terms must have absolute value at least $n$. Without loss of generality, it is the term associated to $S$. Then by the fourth part of Lemma~\ref{shifted-Euler-bounds}, $R_S( d_1,\dots,d_{\rp+\rm}) = O \left( n^{O(1)} q^{n-n/2} \right)$. \end{proof}

\begin{lemma}\label{third-type-bound} The sum \eqref{third-type-of-terms} is $O \left( n^{O(1)} q^{ n /2} \right) $. \end{lemma}

\begin{proof} By Lemma \ref{third-type-range}, the number of nonzero terms in \eqref{third-type-of-terms} is $n^{O(1)}$, and the factor $\prod_{1 \leq i_1< i_2 \leq \rp+\rm} |d_{i_1}-d_{i_2} | $ appearing in each term is $n^{O(1)}$. By Lemma \ref{third-type-individual}, the factor $R_s(d_1,\dots,d_{\rp+\rm})$ appearing in each term is $O \left( n^{O(1)} q^{ n /2} \right)$. So the sum over all terms is $O \left( n^{O(1)} q^{ n /2} \right)$. \end{proof}

Next we handle \eqref{second-type-of-terms}.

\begin{lemma}\label{second-type-bound} The sum \eqref{second-type-of-terms} is $O ( n^{O(1)} q^{n/2})$. \end{lemma}

\begin{proof} Fix $d_1<\dots< d_{\rp+\rm}$ that contribute nontrivially to \eqref{second-type-of-terms}, i.e. \eqref{S-cancellation-range} is violated and $R_S(d_1,\dots, d_{\rp+\rm})$ is nonzero.  By Lemma~\ref{M-support}, the nonvanishing implies that \[ 0 \leq \sum_{i \in S } d_i  - { |S| \choose 2} ,  0 \leq  \sum_{i \notin S}  (n-1-d_i) + { \rp+\rm \choose 2} - {|S| \choose 2},\] so because \eqref{S-cancellation-range} is violated, we must have either  \[ \sum_{i \in S } d_i  - { |S| \choose 2}  \geq n\]  or \[ \sum_{i \notin S}  (n-1-d_i) + { \rp+\rm \choose 2} - {|S| \choose 2}  \geq n .\]  Let $k = \max( \sum_{i \in S } d_i  - { |S| \choose 2} ,  \sum_{i \notin S}  (n-1-d_i) + { \rp+\rm \choose 2} - {|S| \choose 2}  )$. Then there are $k^{O(1)}$ tuples  $d_1,\dots, d_{\rp+\rm}$ leading to a given $k$, each of which has coefficients $O( k^{O(1)} q^{n - k/2} )$ by the second and third parts of Lemma~\ref{shifted-Euler-bounds}, and $ \prod_{1 \leq i_1< i_2 \leq \rp+\rm} |d_{i_1}-d_{i_2} |  = k^{O(1) }$, so in total \eqref{second-type-of-terms} is $\sum_{k \geq n} k^{O(1)} q^{n- k/2} = O ( n^{O(1)} q^{n/2})$. \end{proof}

\begin{lemma}\label{first-type-bound} The sum \eqref{first-type-of-terms} is  $O \left(  n^{\rp+\rm}  C_{\rp+\rm}^{n-1} q^{\frac{n+w}{2}}\right)$. \end{lemma}

\begin{proof} By Lemma~\ref{L-representation-identity}, \[ T ( d_1,d_2+1,\dots, d_{\rp+\rm}+\rp+\rm-1 ) = \pm F\left(  V_{d_1,\dots,d_\rp | d_{\rp+1},\dots, d_{\rp+\rm } }\right ) \] With this renormalized set of coefficients, \eqref{first-type-of-terms} is a sum over tuples $0 \leq d_1 \leq \dots \leq d_{\rp+\rm} \leq n-1$ such that for no $S$ of cardinality congruent to $r$ mod $m$ do we have $0 \leq \sum_{i \in S} d_i, \sum_{i \not \in S} (n-1-d_i) \leq n-1$.  Hence by Lemma~\ref{easy-rep-characterization}, for these tuples $V_{d_1,\dots,d_\rp | d_{\rp+1},\dots, d_{\rp+\rm } }(L_{\univ} )$ does not appear as a summand of $L_{\univ}^{\otimes a} \otimes L_{\univ}^{\vee \otimes b}$ for $0\leq a,b\leq n-1$.

By definition, we have
\[ F\left(  V_{d_1,\dots,d_\rp | d_{\rp+1},\dots, d_{\rp+\rm } }\right ) = \sum_{j \in \mathbb Z} (-1)^j \ \tr(\Frob_q, H^j_c( \Prim_{n,\overline{\mathbb F}_q}, V_{d_1,\dots,d_\rp| d_{\rp+1},\dots,r_{\rp+\rm}}(L_{\univ}(1/2)))) .\]

Because $V_{d_1,\dots,d_\rp | d_{\rp+1},\dots, d_{\rp+\rm } }(L_{\univ} )$ is irreducible, appears as a summand of $\det^{-s} (L_{\univ}) \otimes \bigotimes_{i=1}^{\rp+\rm} \wedge^{d_i}( L_{\univ})$, and does not appear as a summand of $L_{\univ}^{\otimes a} \otimes L_{\univ}^{\vee \otimes b}$ for $0\leq a,b\leq n-1$, we may apply Hypothesis $\operatorname{H}(n,\rp,\rm,w)$ to get $H^j_c( \Prim_{n,\overline{\mathbb F}_q}, V_{d_1,\dots,d_\rp| d_{\rp+1},\dots,r_{\rp+\rm}}(L_{\univ}(1/2)) )=0 $ for $j> n +w$.

Because the $L$-functions $L(s,\chi)$ satisfy the Riemann hypothesis, the eigenvalues of Frobenius on $L_{\univ}$ have size $\sqrt{q}$ (see Definition \ref{L-univ}), and so $L_{\univ}$ is pure of weight $1$. This implies that $L_{\univ}(1/2)$ is pure of weight $0$ and  thus $V_{d_1,\dots,d_\rp| d_{\rp+1},\dots,r_{\rp+\rm}}(L_{\univ}(1/2)) $ is also pure of weight $0$. By Deligne's Riemann hypothesis \cite[Theorem 1]{weil-ii}, all eigenvalues of $\Frob_q$ on $H^j_c( \Prim_{n,\overline{\mathbb F}_q}, V_{d_1,\dots,d_\rp| d_{\rp+1},\dots,r_{\rp+\rm}}(L_{\univ}(1/2)) )$ are $\leq q^{j/2} \leq q^{ \frac{n+w}{2}}$. Thus because the trace of an endomorphism of a vector space is at most the dimension times the size of the greatest eigenvalue,
\[ \left| F\left(  V_{d_1,\dots,d_\rp | d_{\rp+1},\dots, d_{\rp+\rm } }\right )  \right| \leq q^{\frac{n+w}{2}}  \sum_{j \in \mathbb Z} \dim H^j_c( \Prim_{n,\overline{\mathbb F}_q}, V_{d_1,\dots,d_\rp| d_{\rp+1},\dots,r_{\rp+\rm}}(L_{\univ})).\]

Hence \eqref{first-type-of-terms} is at most 
%
%\[  \leq \sum_{\substack{ d_1< \dots <  d_{\rp+\rm} \\ \textrm{\eqref{S-cancellation-range} does not hold for any }S } } \left| T ( d_1,\dots,d_{\rp+\rm} ) \right| \frac{  \prod_{1 \leq i_1< i_2 \leq \rp+\rm} |d_{i_1}-d_{i_2} | }{ \prod_{i=1}^{\rp+\rm} (i-1)!}
%
%
%\left| \sum_{ \substack{0 \leq d_1 \leq \dots \leq d_{\rp+\rm} \leq n-1\\ \exists S, m| |S|-r,  \sum_{i \in S} d_i, \sum_{i\not \in S} (n-1-d_1) \leq n-1} } \sum_{ \sigma \in S_{\rp+\rm} }\sign(\sigma) q^{ \sum_{i=1}^{\rp+\rm} (d_i+ i-1 ) \alpha_{\sigma(i)}  }  (-1)^{ \sum_{i=1}^{\rp+\rm}  d_i + {\rp+\rm \choose 2}   }  F\left(  V_{d_1,\dots,d_\rp | d_{\rp+1},\dots, d_{\rp+\rm } }\right ) \right|\]
%
%\[ \leq  \sum_{ 0 \leq d_1 \leq \dots \leq d_{\rp+\rm} \leq n-1} \left| \sum_{ \sigma \in S_{\rp+\rm} }\sign(\sigma) q^{ \sum_{i=1}^{\rp+\rm} (d_i+ i-1 ) \alpha_{\sigma(i)}  } \right|  \left| F\left(  V_{d_1,\dots,d_\rp | d_{\rp+1},\dots, d_{\rp+\rm } }\right ) \right| \]
\[   \sum_{ 0 \leq d_1 \leq \dots \leq d_{\rp+\rm} \leq n-1}\frac{ \prod_{1 \leq i_1 < i_2 \leq \rp+\rm} | d_{i_2}+i_2 - d_{i_1} -i_1|}{ \prod_{i=1}^{\rp+\rm} (i-1)!}  q^{\frac{n+w}{2}}  \sum_{j \in \mathbb Z} \dim H^j_c( \Prim_{n,\overline{\mathbb F}_q}, V_{d_1,\dots,d_\rp| d_{\rp+1},\dots,r_{\rp+\rm}}(L_{\univ})) .\]

%If we take the equality Lemma~\ref{representation-generating-identity}, which is an identity in the ring $R(GL_{n-1})[q^{\alpha_1},\dots, q^{\alpha_{\rp+\rm} ]$, divide both sides by $\prod_{1\leq i_1< i_2 \leq \rp+\rm} (q^{\alpha_{i_1}}- q^{\alpha_{i_2}}$,  and project onto $R(GL_{n-1})$ by setting $q^{\alpha_i}= -1$, and then project onto $\mathbb Z$,  $R(GL_{n-1})$ \to \mathbb Z$ that sends $V_{d_1,\dots,d_\rp  | d_{\rp+1},\dots, d_{\rp+\rm} }$ to $1$ and all other irreducible representations to $0$, then from the left side \ref{rgi-left} we obtain $(-1)^{{\rp+\rm\choose 2}}$ times the multiplicity of $V_{d_1,\dots,d_\rp} $ in $\bigoplus{ 0\leq e_1, \dots, e_{\rp+\rm} \leq n-1} \det^{-s} \otimes \bigotimes_{i=1}^{\rp+\rm} \wedge^{e_i}$ , and from the right side \ref{rgi-right} we obtain  $(-1)^{{\rp+\rm\choose 2}}$ times the multiplicity of 
%
%we obtain $\prod_{1\leq i_1< i_2 \leq \rp+\rm} (q^{\alpha_{i_1}}- q^{\alpha_{i_2}}$ times the multiplicity of $V_{d_1,\dots,d_\rp} }

By the Weyl character formula and Weyl dimension formula, \[\frac{ \prod_{1 \leq i_1< i_2 \leq \rp+\rm} |d_{i_2}+i_2-d_{i_1}-i_1 |}{ \prod_{i=1}^{\rp+\rm} (i-1)!} = \lim_{\alpha_1,\dots,\alpha_r \to 1}  \frac{ \sum_{\sigma \in S_{\rp+\rm}} \sign(\sigma) q^{( d_i +i-1)\alpha_\sigma(i)}}{ \prod_{1\leq i_1< i_2 \leq \rp+\rm} (q^{\alpha_{i_2}}- q^{\alpha_{i_1}})} \] which, by Lemma~\ref{representation-multiplicity-identity}, is the multiplicity that $V_{d_1,\dots,d_\rp  | d_{\rp+1},\dots, d_{\rp+\rm} }$ appears in \[\bigoplus_{ 0\leq e_1, \dots, e_{\rp+\rm} \leq n-1} \det^{-\rm} \otimes \bigotimes_{i=1}^{\rp+\rm} \wedge^{e_i}.\] Hence \[  \sum_{0 \leq d_1 \leq d_2 \leq \dots \leq d_{\rp+\rm}\leq n-1}   \frac{ \prod_{1 \leq i_1< i_2 \leq \rp+\rm} |d_{i_2}+i_2-d_{i_1} -i_1|}{ \prod_{i=1}^{\rp+\rm} (i-1)!} \sum_j  \dim H^j_c ( \Prim_{n, \overline{\mathbb F}_q},V _{d_1,\dots,d_\rp | d_{\rp+1},\dots, d_{\rp+\rm}} ( L_{\univ})) \] \[\leq \sum_{0 \leq e_1,\dots, e_{\rp+\rm} \leq n-1 } \dim H^j_c ( \Prim_{n, \overline{\mathbb F}_q} , \det(L_{\univ})^{-\rm} \otimes \bigotimes_{i=1}^{\rp+\rm} \wedge^{e_i} (L_{\univ})).\]

The total number of terms here is $n^{\rp+\rm}$. For each term, we apply Lemma~\ref{Betti-number-bound} and use the fact that \[ \det(L_{\univ})^{-\rm} \otimes \bigotimes_{i=1}^{\rp+\rm} \wedge^{e_i} (L_{\univ}) = \bigotimes_{i=1}^\rp \wedge^{e_i}(L_{\univ} )  \otimes \bigotimes_{i=\rm+1}^{\rp+\rm} \wedge^{n-1-e_1} (L_{\univ}^\vee).\] We can assume in each term that $e_1 \leq e_2 \dots \leq e_{\rp+\rm}$, so that $\sum_{i=1}^{\rp} e_i  + \sum_{i=r+1}^{\rp+\rm} (n-1-e_i)  \leq \rp  e_\rp+\rm (n-1-e_{r+1}) \leq \rp e_\rp+\rm(n-1-e_r) \leq (n-1) \max(\rp,\rm).$ This gives exactly the bound $O( C_{\rp+\rm}^{n-1} q^{\frac{n+w}{2}})$ for each term and thus $O( n^{\rp+\rm}C_{\rp+\rm}^{n-1} q^{\frac{n+w}{2}})$ in total.

\end{proof}

\begin{remark} It is unsurprising that, in this proof, dimensions of irreducible representations of $GL_{\rp+\rm}$ appear as multiplicities of irreducible representations of $GL_{n-1}$ in \[\sum_{ 0\leq e_1, \dots, e_{\rp+\rm} \leq n-1} \det^{-\rm} \otimes \bigotimes_{i=1}^{\rp+\rm} \wedge^{e_i},\] because $\otimes \bigotimes_{i=1}^{\rp+\rm} \wedge^{e_i}$ admits a natural action of $GL_{n-1} \times GL_{\rp+\rm}$ as it is the exterior algebra of the tensor product of the standard representations of $GL_{n-1}$ and $GL_{\rp+\rm}$, and this action is preserved after tensoring with the $-\tilde{r}$ power of the determinant of $GL_{n-1}$. This is part of the approach of Bump and Gamburd \cite{BumpGamburd} to moments of the characteristic polynomial of random matrices.\end{remark}

\begin{cor} Assume $n\geq 3$, if $n=3$ that the characteristic of $\mathbb F_q$ is not $2$ or $5$, and if $n=4$ or $5$ that the characteristic of  $\mathbb F_q$ is not $2$.

Assume Hypothesis $\operatorname{H}(n,\rp,\rm,w)$. Let $\alpha_{1},\dots,\alpha_{\rp+\rm}$ be imaginary. Let $C_{\rp,\rm}=\left(\max(\rp,\rm)+2\right)^{\max(\rp,\rm)+1}$  

\[ \frac{1}{ q^{n}- q^{n-1} } \sum_{\chi\in S_{n,q}} \epsilon_\chi^{-\rm} \prod_{i=1}^{\rp+\rm} L(1/2- \alpha_i, \chi) \] \[ =    \sum_{\substack{  S \subseteq \{1,\dots, \rp+\rm\}  \\ m | \rp-|S| }}\mu^{ \frac{\rp- |S|}{m}}  \prod_{i \notin S} q^{ \alpha_i (n-1)} \sum_{ \substack{ f_1,\dots, f_{\rp+\rm} \in \mathbb F_q[T] \\ \textrm{monic} \\ \prod_{i \in S} f_i / \prod_{i\notin S} f_i \in T^{\mathbb Z} }}  \prod_{i\in S} |f_i|^{ -\frac{1}{2} +\alpha_i} \prod_{i \notin S} |f_i|^{ - \frac{1}{2} - \alpha_i} \] \[ + O \left(  q^{\frac{w-n}{2}}  C_{\rp,\rm} ^n \right) . \]

If $n> 2 \max(\rp,\rm)+1$ then we need only the terms where $\rp=|S|$. 

\end{cor}

\begin{proof} The first claim follows from Proposition~\ref{main1} after dividing both sides by $(q^n-q^{n-1}) \prod_{1 \leq i_1< i_2 \leq \rp+\rm} (q^{\alpha_{i_1}} - q^{\alpha_{i_2}}) $.

The second claim follows from Lemma~\ref{m-estimate}, because then $m> \max(\rp,\rm)$.

\end{proof} 

In particular, the second claim is Theorem \ref{main}.

\section{Verification of the hypothesis in special cases}

\begin{lemma}\label{key-verification}

Let $\mathcal F$ be an irreducible lisse $\mathbb Q_\ell$-sheaf on $\Prim_{n, \overline{\mathbb F}_q}$ that appears as a summand of \[L_{\univ} ^{\otimes a} \otimes L_{\univ}^{\vee \otimes b} \] for some $a \geq n > b$.
Then \[H^j_c( \Prim_n, \mathcal F)=0\] for $j> n+b + \left\lfloor \frac{a}{p} \right\rfloor -\left  \lfloor \frac{n}{p}\right \rfloor +1 $.  \end{lemma}

\begin{proof}To do this, we will first examine the cohomology of the space $Z_{n,a,b}$ defined in Subsection~\ref{betti-bounds}. We relate them to the spaces $X_{a,n, (c_1,\dots,c_n)}$ of \cite{me}, defined as the subspace of $\mathbb A^a$ with variables $x_1,\dots,x_a$ satisfying the system of $n$ equations $\prod_{i=1}^a (1-T a_i) = 1+ c_1 T + \dots + c_n T^n \mod T^{n+1}$. 

The key fact that we will derive from \cite{me} is that the $S_a$ action on the cohomology of $X_{a,n, (c_1,\dots,c_n)}$  is trivial in degrees greater than $a-n +  \left\lfloor \frac{a}{p} \right\rfloor - \left\lfloor \frac{n}{p} \right\rfloor  +1 $. To see this, first note that the cohomology in degree $>2(a-n)$ vanishes for dimension reasons. Second, the cohomology in degree $<2(a-n)$ is handled by the first part of \cite[Proposition 2.5]{me}.  Finally the cohomology in degree is exactly $2(a-n)$ is handled by the second part of \cite[Proposition 2.5]{me}, because our condition that the degree is greater than  $a-n +  \left\lfloor \frac{a}{p} \right\rfloor - \left\lfloor \frac{n}{p} \right\rfloor  +1 $ is then equivalent to the condition  $a-n >   \left\lfloor \frac{a}{p} \right\rfloor - \left\lfloor \frac{n}{p} \right\rfloor  +1 $ of \cite[Proposition 2.5]{me}.

The space $Z_{n,a,b}$ admits a map to $\mathbb A^b$ by projection onto the last $b$ coordinates, and the fibers of this map are the spaces $X_{a,n, (c_1,\dots,c_n)}$, where $c_1,\dots,c_m$ are the coefficients to the product of the last $b$ linear factors. So by the proper base change theorem, the fact that the cohomological dimension of $\mathbb A^b$ is $2b$, and our key fact, the action of $S_a$ on the cohomology of $Z_{n,a,b}$ is trivial in degrees greater than $ a + 2b-n \left\lfloor \frac{a}{p} \right \rfloor - \left\lfloor \frac{n}{p} \right\rfloor  +1 $. 

Hence by Lemma \ref{Fourier-comparison}, the  $S_a$ action on \[H^{j}_c \left(\Witt_{n,\overline{\mathbb F}_q}, \left( R (pr_2)_! \mathcal L_{\univ} \right)^{ \otimes a} \otimes  \left( R (pr_2)_! \mathcal L_{\univ}^\vee \right)^{\otimes b}\right) \]  is trivial whenever $j>a + 2b-n \left\lfloor \frac{a}{p}\right \rfloor -\left \lfloor \frac{n}{p} \right\rfloor  +1 $. Applying this to $n-1$, we can see that the $S_a$ action on  \[H^{j-1}_c \left(\Witt_{n-1,\overline{\mathbb F}_q}, \left( R (pr_2)_! \mathcal L_{\univ} \right)^{ \otimes a} \otimes  \left( R (pr_2)_! \mathcal L_{\univ}^\vee \right)^{\otimes b}\right) \] is trivial whenever \[j-1>  a+2b + n-1  + \left\lfloor \frac{a}{p} \right\rfloor - \left \lfloor \frac{n-1 }{p}\right \rfloor +1  \geq \left(  a+2b + n + \left\lfloor \frac{a}{p} \right\rfloor -  \left\lfloor \frac{n }{p}\right \rfloor +1 \right)-1 .\]

By the long exact sequence of Lemma~\ref{witt-excision}, the action of $S_a$ on \[ H^j_c \left(\Prim_{n, \overline{\mathbb F}_q},  \bigotimes_{i=1}^{a}  L_{\univ} [-1]  \otimes \bigotimes_{i=1}^{b} L_{\univ} ^\vee [-1]  \right) \] is trivial for $j>a +2b +n+ \lfloor \frac{a}{p} \rfloor -  \lfloor \frac{n}{p} \rfloor +1 $. Hence by shifting, the action of $S_a$ on \[ H^j_c \left(\Prim_{n, \overline{\mathbb F}_q},  \bigotimes_{i=1}^{a}  L_{\univ}  \otimes \bigotimes_{i=1}^{b} L_{\univ} ^\vee \right)  \] factors through the sign character  if $j> n+b + \lfloor \frac{a}{p} \rfloor -  \lfloor \frac{n}{p} \rfloor +1 $. But the sign-equivariant part of $\bigotimes_{i=1}^{a}  L_{\univ}  $ is $\wedge^a L_{\univ}$, which vanishes, so the sign-equivariant part of the cohomology vanishes as well. Thus in fact \[ H^j_c \left(\Prim_{n, \overline{\mathbb F}_q},  \bigotimes_{i=1}^{a}  L_{\univ}  \otimes \bigotimes_{i=1}^{b} L_{\univ} ^\vee \right) =0 \]  if $j> n+b + \lfloor \frac{a}{p} \rfloor -  \lfloor \frac{n}{p} \rfloor +1 $. Hence the same is true for any summand of $\bigotimes_{i=1}^{a}  L_{\univ}  \otimes \bigotimes_{i=1}^{b} L_{\univ} ^\vee$, such as $\mathcal F$. \end{proof}

\begin{lemma} Hypothesis $\operatorname{H}(n,\rp, 0 ,\left \lfloor \frac{ (n-1)\rp }{p} \right \rfloor-\left \lfloor \frac{n}{p} \right\rfloor +1)$ is satisfied for any $n,\rp$.  \end{lemma}

\begin{proof} Any sheaf $\mathcal F$ that is a summand of $  \bigotimes_{i=1}^{\rp} \wedge^{d_i} (L_{\univ} )$ for some $0 \leq d_1,\dots,d_{\rp}\leq n-1$ is a summand of $L_{\univ}^{\otimes (\sum_{i=1}^r d_\rp)}$. If $\sum_{i=1}^\rp d_i < n$ then the condition of Hypothesis~\ref{hypo} is not satisfied, so it is vacuously true. Otherwise, we apply Lemma~\ref{key-verification} with $a =\sum_{i=1}^\rp{d_i} \leq \rp (n-1)$ and $b=0$. \end{proof}

However, combining this with Theorem~\ref{main1} would simply recover \cite[Theorems 1.2 and 1.3]{me}.

\begin{lemma}\label{power-savings} Hypothesis $\operatorname{H}(n,\rp, 1, w= n +1 - \frac{p-2\rp}{p\rp} n)$ is satisfied for any $n,\rp$. \end{lemma}

 \begin{proof} We may assume $p>2\rp$ as if $ p \leq 2\rp $ then the claim follows immediately from the fact that $H^j ( \Prim_{n, \overline{\mathbb F}_q}, \mathcal F)=0$ for $j> 2n$ because $\dim \Prim_n = n$.
 
 Let $\mathcal F$ be a summand of $\det^{-1} (L_{\univ})\otimes \bigotimes_{i=1}^{\rp+1} \wedge^{d_i} (L_{\univ} )$ for some $0 \leq d_1,\dots,d_{\rp+1} \leq n-1$. Without loss of generality, $0 \leq d_1 \leq \dots \leq d_{\rp}\leq d_{\rp+1} \leq n-1$. Then because $\det^{-1} (L_{\univ})\otimes \wedge^{d_{\rp+1}} (L_{\univ}) = \wedge^{n-1-\rp} (L_{\univ}^\vee)$, $\mathcal F$ is a summand of $L_{\univ} ^{\otimes (\sum_{i=1}^\rp  d_i)} \otimes L_{\univ}^{\vee \otimes (n-1-d_{\rp+1}) }$. If $\sum_{i=1}^\rp d_i < n$ then Hypothesis $\operatorname{H}(n,\rp,1,n +1 - \frac{p-2\rp}{p\rp} n)$ is vacuously true. Otherwise, $d_{\rp+1} \geq d_\rp \geq \frac{n}{\rp} $ and we apply Lemma~\ref{key-verification} to see that the cohomology groups vanish for \[ j> n + (n-1-d_{\rp+1}) + \left \lfloor \frac{ \sum_{i=1}^{\rp} d_\rp }{p} \right\rfloor  - \left \lfloor \frac{n}{p}\right \rfloor + 1,\] and we have \[(n-1-d_{\rp+1}) +\left \lfloor \frac{ \sum_{i=1}^r d_\rp }{p}\right \rfloor  -\left \lfloor \frac{n}{p} \right\rfloor + 1 \leq  n -1 - d_{\rp+1} + \frac{\sum_{i=1}^r d_\rp}{p} - \frac{n}{p} + 2 \] \[ \leq n - 1 - \frac{p-\rp}{p}  d_{\rp+1} - \frac{n}{p} + 2 \leq  n -1 - \frac{p-\rp}{p\rp} n - \frac{n}{p}  +2 = n  +1 - \frac{p - 2\rp}{p\rp} n\] as desired.

\end{proof} 

We could apply the same technique with $\rp,\rm \geq 2$ but we would not obtain a nontrivial bound this way.

\begin{cor}\label{combined-result} Let $n,\rp$ be natural numbers and $\mathbb F_q$ a finite field of characteristic $p$. Assume also that $n> 2 \rp+1$ and if $n=4$ or $5$ that the characteristic of $\mathbb F_q$ is not $2$. Let $C_{\rp,1} = (2 +\rp)^{\rp+1}$.  Let $\alpha_1,\dots,\alpha_{\rp+1}$ be imaginary numbers. Let $\epsilon_\chi$ be the $\epsilon$-factor of $L(\chi)$.   Then

\[ \frac{1}{ (q^{n}- q^{n-1}) } \sum_{\chi \in S_{n,q}}  \epsilon_\chi^{-1} \prod_{i=1}^{\rp+1} L(1/2- \alpha_i, \chi)\]   \begin{equation}\label{combined-result-formula} =    \sum_{j=1}^{\rp+1} q^{\alpha_j (n-1)} \left(\frac{1}{1- q^{-\frac{1}{2} - \alpha_j}  }\prod_{i \neq j}  \frac{ 1- q^{-1 + \alpha_i - \alpha_j}}{ (1-q^{-\frac{1}{2}+ \alpha_i}) (1-q^{\alpha_i-\alpha_j})}\right) + O \left(  \sqrt{q}  \left(  q^{ - \frac{p-2\rp}{2p\rp} } C_{\rp,1} \right) ^n  n^{\rp+1} \right) . \end{equation}

\end{cor}

\begin{proof} By plugging Lemma~\ref{power-savings} into Theorem~\ref{main}, we obtain an identical formula to \eqref{combined-result-formula}, except that the main term is 

\[    \sum_{j=1}^{\rp+1}  q^{ \alpha_j (n-1)} \sum_{ \substack{ f_1,\dots, f_{\rp+1} \in \mathbb F_q[T] \\ \textrm{monic} \\ (\prod_{i\neq j}  f_i)  /f_j \in T^{\mathbb Z} }}  ( \prod_{i\neq j } |f_i|^{ -\frac{1}{2} +\alpha_i})   |f_j|^{ - \frac{1}{2} - \alpha_j} .\]

But by factoring out powers of $T$, and then noting that $f_j$ is uniquely determined by the other variables, we obtain  \[ \sum_{ \substack{ f_1,\dots, f_{\rp+1} \in \mathbb F_q[T] \\ \textrm{monic} \\ (\prod_{i\neq j}  f_i)  /f_j \in T^{\mathbb Z} }} ( \prod_{i\neq j } |f_i|^{ -\frac{1}{2} +\alpha_i})   |f_j|^{ - \frac{1}{2} - \alpha_j}  \]

\[ =\left(\prod_{i \neq j} \frac{1}{ 1-q^{-\frac{1}{2}+ \alpha_i}}\right)  \frac{1}{1- q^{-\frac{1}{2} - \alpha_j}  }\sum_{ \substack{ f_1,\dots, f_{\rp+1} \in \mathbb F_q[T] \\ \textrm{monic} \\ (\prod_{i\neq j}  f_i)  = f_j \\ f_1,\dots, f_{\rp+1} \textrm{ prime to } T}  }( \prod_{i\neq j } |f_i|^{ -\frac{1}{2} +\alpha_i})   |f_j|^{ - \frac{1}{2} - \alpha_j} \]

\[ =\left(\prod_{i \neq j} \frac{1}{ 1-q^{-\frac{1}{2}+ \alpha_i}}\right)  \frac{1}{1- q^{-\frac{1}{2} - \alpha_j}  }\sum_{ \substack{ f_1,\dots, f_{j-1}, f_{j+1}, \dots  f_{\rp+1} \in \mathbb F_q[T] \\ \textrm{monic} \\ f_1,\dots, f_{j-1},f_{j+1},\dots f_{\rp+1} \textrm{ prime to } T} } ( \prod_{i\neq j } |f_i|^{ -1  +\alpha_i - \alpha_j })\]

\[ =\left(\prod_{i \neq j} \frac{1}{ 1-q^{-\frac{1}{2}+ \alpha_i}}\right)  \frac{1}{1- q^{-\frac{1}{2} - \alpha_j}  }\prod_{i \neq j}  \frac{ 1- q^{-1 + \alpha_i - \alpha_j}}{ 1-q^{\alpha_i-\alpha_j}} .\]

Plugging this in, we get \eqref{combined-result-formula}. 
 
\end{proof}

\section{Spaces defined by Hast and Matei}

The results in this section are not directly related to the main results of this paper, but use similar techniques. In it, we recover by a more direct geometric argument calculations by Hast and Matei of certain cohomology groups of certain spaces. We also sketch how more cohomology groups might be computed conditionally on Hypothesis \ref{hypo}.

Using Katz's equidistribution results for the $L$-functions of Dirichlet characters \cite{WVQKR}, and performing a Fourier transform over the group of Dirichlet characters, Rodgers calculated the variance of certain arithmetic functions in the $q \to \infty$ limit \cite{Rodgers}. Hast and Matei defined schemes in a natural way so that their cohomology would control these variances, and then used Rodgers' calculations to control the top nonvanishing cohomology groups \cite{HastMatei}. This has the effect of starting with geometric information (the monodromy calculations of Katz), proceeding to numerical information, and deriving geometric information again. Using the Witt vector Fourier transform calculation in Lemma~\ref{Fourier-comparison}, it is possible to avoid the numerical step and reason entirely geometrically. This answers a question asked in \cite{HastMatei}.

\begin{lemma}\label{hm-sheaf-calculation} Let $n$ and $m$ be natural numbers with $m>n$.  Assume that $n\geq 3$, and, if $n=3$, that $p\neq 2,5$.

Then the cohomology group \[H^j_c\left( \Witt_{n, \overline{\mathbb F}_q},  \left( R (pr_2)_! \mathcal L_{\univ} \right)^{ \otimes m} \otimes  \left( R (pr_2)_! \mathcal L_{\univ}^\vee \right)^{\otimes m} \right) \] for $j \geq 2n+2m$ are described as follows: \begin{itemize}

\item If $j> 2n+2m$ and $j\neq 4m$, then this cohomology group vanishes.

\item If $j=4m$, this cohomology group is $\mathbb Q_\ell (-2m)$ with the trivial $S_m \times S_m$ action.

\item If $j=2n+2m$, this cohomology group is $\operatorname{Hom}_{GL_{n-1}} (V^\otimes m, V^{\otimes m}) (-n-m)$ where $V$ is the $n-1$-dimensional standard representation of $GL_{n-1}$ over $\mathbb Q_\ell$, with $S_m \times S_m$ acting by permuting the factors. \end{itemize}\end{lemma}

\begin{proof} We stratify $\Witt_n$ into, first, the open subset $\Prim_n = \Witt_n \setminus \Witt_{n-1}$, second, $\Witt_{n-1} \setminus \{0\}$, and third, the point $\{0\}$. We will calculate the cohomology independently on each of the three sets, then combine the information.

On $\Witt_n \setminus \{0\}$, $R (pr_2)_! \mathcal L_{\univ} $ and $R (pr_2)_! \mathcal L_{\univ}^\vee$ are supported  in degree one, so $\left( R (pr_2)_! \mathcal L_{\univ} \right)^{ \otimes m} \otimes  \left( R (pr_2)_! \mathcal L_{\univ}^\vee \right)^{\otimes m}$ is supported in degree $2m$. Thus the cohomology of $\Witt_{n-1}$ with that complex is supported in degree $\leq 2(n-1) + 2m$. Hence by excision, the natural map \[ H^j _c \left( \Prim_{n,\overline{\mathbb F}_q }, \left( R (pr_2)_! \mathcal L_{\univ} \right)^{ \otimes m} \otimes  \left( R (pr_2)_! \mathcal L_{\univ}^\vee \right)^{\otimes m}\right)\] \[\to  H^j _c \left( \Witt_{n,\overline{\mathbb F}_q} \setminus \{0\}), \left( R (pr_2)_! \mathcal L_{\univ} \right)^{ \otimes m} \otimes  \left( R (pr_2)_! \mathcal L_{\univ}^\vee \right)^{\otimes m}\right)\]  is an isomorphism in degrees $> 2(n-1) + 2m+1$. 

Again applying the cohomological dimension bound,  $H^j _c \left( \Prim{n,\overline{\mathbb F}_q} , \left( R (pr_2)_! \mathcal L_{\univ} \right)^{ \otimes m} \otimes  \left( R (pr_2)_! \mathcal L_{\univ}^\vee \right)^{\otimes m}\right)$ is supported in degrees $\leq 2n+2m$. Furthermore

\[ 
H^{2m+2n} _c \left( \Prim{n,\overline{\mathbb F}_q} , \left( R (pr_2)_! \mathcal L_{\univ} \right)^{ \otimes m} \otimes  \left( R (pr_2)_! \mathcal L_{\univ}^\vee \right)^{\otimes m}\right) = H^{2n} \left( \Prim_{n,\overline{\mathbb F}_q},  L_{\univ}^{\otimes m}   \otimes L_{\univ}^{\vee \otimes m } (-m) \right)\] 
with the Tate twist because $R^1 (pr_2)_! \mathcal L_{\univ}^\vee = L_{\univ}^\vee(-1))$.  Applying Poincare duality, we have
\[  H^{2n} \left( \Prim_{n,\overline{\mathbb F}_q},  L_{\univ}^{\otimes m}   \otimes L_{\univ}^{\vee \otimes m } (-m) \right) = H^0 \left(\Prim_{n,\overline{\mathbb F}_q},  L_{\univ}^{\vee \otimes m}   \otimes L_{\univ}^{ \otimes m }  \right)^\vee  (-n-m)\] \[  = \Hom( L_{\univ}^{\otimes m}, L_{\univ}^{\otimes m} )^\vee  (-n-m) = \Hom (L_{\univ}^{\otimes m}, L_{\univ}^{\otimes m}) (-n-m).\]

By \cite[Theorem 5.1]{WVQKR}, the geometric monodromy group of $L_{\univ}$ is contained between $SL_{n-1}$ and $GL_{n-1}$. Letting $V$ be the standard representation of $GL_{n-1}$, we have \[ \Hom_{GL_{n-1}} ( V^{\otimes m}, V^{\otimes m}) \subseteq  \Hom (L_{\univ}^{\otimes m}, L_{\univ}^{\otimes m})  \subseteq \Hom_{SL_{n-1}}(V^{\otimes m}, V^{\otimes m}).\] Because the center of $GL_{n-1}$ acts by scalars on $V^{\otimes m}$, every $SL_{n-1}$-equivariant endomorphism is also $GL_{n-1}$-equivariant, so all three of these vector spaces are equal, and \[H^{2m+2n} \left( \Witt_{n,\overline{\mathbb F}_q}\setminus \{0\}, \left( R (pr_2)_! \mathcal L_{\univ} \right)^{ \otimes m} \otimes  \left( R (pr_2)_! \mathcal L_{\univ}^\vee \right)^{\otimes m}\right) = \operatorname{Hom}_{GL_{n-1}} (V^\otimes m, V^{\otimes m}) (-n-m).\]

 Observe that at the point $0$,  $R (pr_2)_! \mathcal L_{\univ} $ and $R (pr_2)_! \mathcal L_{\univ}^\vee$ are one-dimensional vector spaces in degree $2$, so \[\left( R (pr_2)_! \mathcal L_{\univ} \right)^{ \otimes m} \otimes  \left( R (pr_2)_! \mathcal L_{\univ}^\vee \right)^{\otimes m}= \mathbb Q_\ell [4m](-2m).\] Thus \[H^j_c\left( \{0\},  \left( R (pr_2)_! \mathcal L_{\univ} \right)^{ \otimes m} \otimes  \left( R (pr_2)_! \mathcal L_{\univ}^\vee \right)^{\otimes m} \right)\] is $\mathbb Q_\ell(-2m)$ if $j=4m$ and $0$ otherwise.
 
 We now apply the excision exact sequence to calculate the cohomology of $\Witt_n$. In degrees at least $2n+2m$, the only nonvanishing terms are in degrees $2n+2m$ and $4m$. As they are both in even degrees the connecting homomorphism between them vanishes, and thus the cohomology of the total space is simply the sum of the contributions from $\{0\}$ and $\Witt_n - \{0\}$, as stated. (In general, it could be an extension, but since $m>n$ they are in different degrees.)
\end{proof}

\begin{cor} Let $n$ and $m$ be natural numbers with $m>n$.  Assume that $n\geq 3$, and, if $n=3$, that $p\neq 2,5$. Then we have

\begin{itemize} 
\item If $j> 2m$ and $j\neq 4m$, then $H^j_c (Z_{n,m,m, \overline{\mathbb F}_q},\mathbb Q_\ell) =0$.

\item If $j=4m-2n$, then $H^j_c (Z_{n,m,m, \overline{\mathbb F}_q},\mathbb Q_\ell)= \mathbb Q_\ell (-2m)$ with the trivial $S_m \times S_m$ action.

\item If $j=2m$, then $H^j_c (Z_{n,m,m, \overline{\mathbb F}_q},\mathbb Q_\ell)= \operatorname{Hom}_{GL_{n-1}} (V^\otimes m, V^{\otimes m}) (-n-m)$ where $V$ is the $n-1$-dimensional standard representation of $GL_{n-1}$ over $\mathbb Q_\ell$, with $S_m \times S_m$ acting by permuting the factors. \end{itemize}\end{cor}

\begin{proof} This follows immediately from Lemmas~\ref{Fourier-comparison} and~\ref{hm-sheaf-calculation}. \end{proof}

\begin{prop}Let $n$ and $m$ be natural numbers with $m>n$.  

 Let $X_{2, n,m}$ be the subset of $\mathbb P^{2m-1}$ with projective coordinates $(a_{1},\dots,a_{m} ,b_{1},\dots,b_{m})$ such that $\prod_{i=1}^{m_1} ( 1 - a_ix )  \equiv \prod_{i=1}^{m_2} (1-b_i x) \mod x^{n+1}$ (matching the definition of \cite[\S2]{HastMatei}).   
 
 Assume that $n\geq 3$, and, if $n=3$, that $p\neq 2,5$. Then we have
 
 \begin{itemize} 
\item If $j> 4m-2m-2$ or $j$ is odd and $j>2m-2$, then $H^j_c (Z_{n,m,m, \overline{\mathbb F}_q},\mathbb Q_\ell) =0$.

\item If $j>2m-2$ is even, then $H^j_c (Z_{n,m,m, \overline{\mathbb F}_q},\mathbb Q_\ell)= \mathbb Q_\ell \left(-\frac{j}{2}\right)$ with the trivial $S_m \times S_m$ action.

\item If $j=2m-2$, then $H^j_c (Z_{n,m,m, \overline{\mathbb F}_q},\mathbb Q_\ell)$ is an extension of $\mathbb Q_\ell(1-m) $ by $ \operatorname{Hom}_{GL_{n-1}} (V^\otimes m, V^{\otimes m}) (1-n-m)$ where $V$ is the $n-1$-dimensional standard representation of $GL_{n-1}$ over $\mathbb Q_\ell$, with $S_m \times S_m$ acting by permuting the factors. \end{itemize}

\end{prop} 

This matches the description of \cite[Theorem A]{HastMatei} by a calculation in Schur-Weyl duality. 

\begin{proof} By comparing definitions, we see that $Z_{n,m,m}$ is the affine cone on $X_{2,n,m}$ (which is called $Y_{2,n,m}$ in \cite[\S2]{HastMatei}).  By excision, our calculations of the cohomology of $Z_{n,m,m}$ hold without modification for $Z_{n,m,m} \neq 0$.

Applying the Leray spectral sequence to the projection $Y_{2,n,m}- \{0\} \to X_{2,n,m}$, we obtain a long exact sequence (dropping subscripts for compactness) \[H^{j+1}_c ( X, \mathbb Q_\ell) \to H^{j+2}_c (Z-\{0\} , \mathbb Q_\ell) \to H^{j}_c(X, \mathbb Q_\ell(-1)) \to H^{j+2}_c (X, \mathbb Q_\ell) \to H^{j+3}_c(Z-\{0\} ,\mathbb Q_\ell).\] 

We verify our description of $H^j_c( X_{2,n,m},\mathbb Q_\ell)$ by descending induction on $j$, starting from $4m-2n-2$, which is twice the dimension of $X_{2,n,m}$ and beyond which we know its cohomolog vanishes. In particular, we may assume that the cohomologies of $X$ and $Z$ vanish in all odd degrees greater than $j$. Because of this, when $j$ is odd, the exact sequence reduces to \[ 0 \to H^{j}_c(X_{2,n,m}, \mathbb Q_\ell(-1)) \to 0,\] verifying the induction step, and when $j$ is even, it reduces to \[ 0 \to H^{j+2}_c (Z_{n,m,m}-\{0\} , \mathbb Q_\ell)(1) \to H^{j}_c(X_{2,n,m}, \mathbb Q_\ell) \to H^{j+2}_c (X_{2,n,m}, \mathbb Q_\ell)(1) \to 0,\] again quickly verifying the induction step.

\end{proof}

We sketch in addition how Hypothesis \ref{hypo} for $c=0$ and $\rp,\rm,n$ arbitrary could potentially be used to calculate the cohomology of $Z_{n,m,m}$ (and hence $X_{2,n,m})$ in degrees greater than $2m-n$. One uses Lemma~\ref{Fourier-comparison}, and expresses the right side by iterated excision as arising by a spectral sequence from the cohomologies of $\Prim_d$, $0 \leq d \leq n$, with cohomology in $ L_{\univ}^{\otimes m} \otimes L_{\univ}^{\vee \otimes m}$. One then decomposes into irreducible representations of $GL_{d-1}$ and discards those that do not satisfy the criterion of Hypothesis~\ref{hypo}. For the remainder, one can calculate the cohomology using Lemma~\ref{representation-generating-identity} to reduce the calculation to the cohomology of tensor products of wedge powers. Using Lemma~\ref{Fourier-comparison} again, we obtain the cohomology of a moduli space of tuples of polynomials of fixed degrees whose product is equal to another tuple of polynomials of fixed degrees. If the cohomology of these spaces can be calculated explicitly, then the cohomology of the Hast-Matei varieties can be calculated, conditionally on the hypothesis.

\end{document}